\documentclass[12pt,draft]{article}
\usepackage[utf8]{inputenc}
\usepackage[total={5.5in, 9in}]{geometry}

\usepackage[page]{appendix}

\usepackage{amsrefs} 
\usepackage{float}
\usepackage{amsmath}
\usepackage{amsthm}
\usepackage{amssymb}
\usepackage{tikz}
\usetikzlibrary{chains,fit,shapes,calc,matrix}

\theoremstyle{definition}
\newtheorem{theorem}{Theorem}[section]
\newtheorem{definition}[theorem]{Definition}
\newtheorem{proposition}[theorem]{Proposition}

\newtheorem{corollary}[theorem]{Corollary}

\newtheorem{lemma}[theorem]{Lemma}

\title{Banach's theorem in higher order reverse mathematics}

\author{
Jeffry L. Hirst\thanks{Department of Mathematical Sciences, Appalachian State University,
Walker Hall, Boone, NC 28608. \textit{Email:}~\texttt{hirstjl@appstate.edu} \hfill Version:  count-20230309sub}%
 \and %
Carl Mummert\thanks{Department of Computer and Information Technology, Marshall University, 1 John Marshall Drive, Huntington, WV 25755. \textit{Email:}~\texttt{mummertc@marshall.edu}}%
}

\date{March 8, 2023}

\newcommand{\RCA}{\mathsf{RCA}}
\newcommand{\WKL}{\mathsf{WKL}}

\newcommand{\QF}{\mathsf{QF}}
\newcommand{\AC}{\text{-}\mathsf{AC}}

\newcommand{\RCAo}{\mathsf{RCA}_0}
\newcommand{\WKLo}{\mathsf{WKL}_0}
\newcommand{\ACAo}{\mathsf{ACA}_0}

\newcommand{\rcaw}{\RCA_0^\omega}

\newcommand{\setN}{\mathbb{N}}
\newcommand{\C}{\mathsf{C}}

\newcommand{\Cenum}{\C_\text{enum}}
\newcommand{\Cinj}{\C_\text{inj}}
\newcommand{\Cbij}{\C_\text{bij}}
\newcommand{\cat}{{^\smallfrown}}

\newcommand{\princb}[1]{({\sf{B}}_{#1})}
\newcommand{\prince}{(\exists^2 )}
\newcommand{\princmuo}{(\mu_0)}
\newcommand{\princmu}{(\mu )}

\newcommand{\lpo}{{\sf{LPO}}}
\newcommand{\llpo}{{\sf{LLPO}}}
\newcommand{\llpomin}{{\sf{LLPOmin}}}

\renewcommand{\phi}{\varphi}

\newcommand{\lh}{{\text{length}}}

\newcommand{\wlt}{\le_{\sf W}}

\mathchardef\mhyphen="2D

\begin{document}

\maketitle

\begin{abstract}
In this paper, methods of second order and higher order reverse mathematics are applied to versions of a theorem
of Banach that extends the Schroeder--Bernstein theorem.  Some additional results address statements in
higher order arithmetic formalizing the uncountability of the power set of the natural numbers.  In general, the
formalizations of higher order principles here have a Skolemized form asserting the existence of functionals that
solve problems uniformly.  This facilitates proofs of reversals in axiom systems with restricted choice.
\end{abstract}


\section{Introduction}

\vskip -5.5in
{\small {This is a pre-review draft.  After review, section and theorem numbering changed and the last section was substantially modified.
The final official version of the paper is published in Computability {\bf 12} (2023) 203-225, DOI 10.3233/COM-230453.  That
version is freely available to CiE members and can be purchased by others.}}
\vskip 5.3in

The Schroeder--Bernstein theorem is perhaps the best known result about cardinality. In full generality, it states that if $A$ and $B$ are sets, there is an injection $f \colon A \to B$, and there is an injection $g \colon B \to A$, then there is a bijection from $A$ to~$B$.
Unfortunately, this theorem is not ideal for reverse mathematics analysis.  If we add the assumption that $A, B \subseteq \setN$, the result is computationally trivial: whenever $A, B \subseteq \setN$  have the same cardinality, there is an $(A \oplus B)$-computable bijection between them.

In higher order reverse mathematics, we might consider the case where $A, B \subseteq 2^\setN$ or $A,B \subseteq \setN^\setN$. In this setting, the Schroeder--Bernstein theorem is no longer trivial. 
 However, because the theorem does not postulate any relationship between the bijection being constructed and the original two injections, obtaining reversals presents a challenge. 

Our focus is a classical theorem of Banach~\cite{banach} from 1924 more suited to reverse mathematical analysis. Banach argued this theorem captures the essence of proofs of the Schroeder--Bernstein theorem, such as the well known proof by Julius K{\H o}nig.

\begin{theorem}[Banach]
  If $A$ and $B$ are sets, $f\colon A \to B$ is an injection, and
  $g\colon B \to A$ is an injection, there 
  are decompositions $A = A_1 \cup A_2$ and $B = B_1 \cup B_2$ such that
  $A_1 \cap A_2 = \emptyset$, $B_1 \cap B_2 = \emptyset$, $f(A_1) = B_1$,
  and $g(B_2) = A_2$.
\end{theorem}

Restating this in terms of the existence of a bijection gives a corollary
that strengthens the Schroeder--Bernstein theorem, which we will also call Banach's Theorem.

\begin{corollary}\label{banachcor}
  If $f$ is an injection from a set $A$ to a set $B$, and
  $g$ is an injection from $B$ to  $A$, there
is a bijection $h \colon A \to B$ such that, whenever $h(a) = b$,
either $f(a) = b$ or $g(b) = a$.
\end{corollary}

A brief history of Banach's Theorem and the Schr\"oder--Bernstein theorem is given by Remmel~\cite{remmel}*{Introduction}. An analysis of Banach's Theorem for subsets of $\setN$,
using subsystems of second order arithmetic, appears in Hirst's thesis~\cite{hirstthesis}*{\S3.2} and a related article~\cite{hirstmt}.  That development uses symmetric marriage theorems to prove the following second order arithmetic results.

\begin{theorem}[\cite{hirstmt}*{Theorem 4.1}]\label{t:hirst1}
$\RCAo$ proves the following are equivalent:
\begin{enumerate}
\item $\ACAo$.
\item (Countable Banach's Theorem)  Let $f\colon \setN \to \setN$ and $g\colon\setN \to \setN$ be injections. Then there is a bijection $h\colon \setN \to \setN$ such that for all $m$ and $n$, $h(n)=m$ implies either $f(n)=m$ or $g(m)=n$.
\end{enumerate}
\end{theorem}

\begin{theorem}[\cite{hirstmt}*{Theorem~4.2}]\label{t:hirst2}
$\RCAo$ proves the following are equivalent:
\begin{enumerate}
\item $\WKLo $.
\item (Bounded Countable Banach's Theorem)  Let $f \colon \setN \to \setN$ and $g\colon\setN \to \setN$ be injections such that the ranges of $f$ and $g$ exist.
Then there is a bijection $h\colon \setN \to \setN$ such that for all $m$ and $n$, $h(n)=m$ implies either $f(n)=m$ or $g(m)=n$.
\end{enumerate}
\end{theorem}

In this paper, we use methods from higher order reverse mathematics to study the \emph{uniformity} of results like these. We are interested not only in the existence of the bijection $h$, but also whether there is a functional that can produce $h$ uniformly from $f$ and~$g$. This question of uniformity is purely higher order, and cannot be expressed directly in second order reverse mathematics.   To study this uniformity, we examine Skolemized versions of theorems. For example, instead of examining Banach's Theorem
in a form such as
\[
(\forall f, g, A, B)(\exists h)\, \Phi(f,g,A,B,h),
\]
we consider the form
\[
(\exists H)(\forall f, g, A, B)\,\Phi(f,g,A,B,H(f,g,A,B)).
\]

Both versions of a theorem are of interest, of course, and the latter always follows from the former if we assume sufficient choice principles. We are interested in the Skolemized forms because they represent a particular kind of uniformity, and we typically do not assume enough choice to derive them directly from the un-Skolemized form. As discussed in Section~\ref{sec:realize}, this is a different kind of uniformity than Weihrauch reducibility.

Section~\ref{sec:count} begins with a survey of reverse mathematics results on countability.   Sections \ref{sec:bounding} and \ref{sec:realize} present a number of supporting lemmas to prepare for the analysis of Banach's theorem. Section~\ref{sec:banachn} examines Theorems \ref{t:hirst1} and~\ref{t:hirst2} from the viewpoint of Skolemized uniformity.  Section~\ref{sec:banachcompact} extends the study of Banach's Theorem to subsets of $2^\setN$ and, more generally, subsets of compact metric spaces.


\subsection{Formal theories}\label{sec:formal}

This work relies on several well studied systems of second order arithmetic and higher order arithmetic. Simpson~\cite{sim09} and Dzhafarov and Mummert~\cite{MR4472209} provide thorough references for reverse mathematics. Kohlenbach~\cite{koh05} provides a reference for higher order reverse mathematics. We follow Kohlenbach's definitions of higher order systems throughout this paper, noting any exceptions explicitly.

For the purposes of higher order reverse mathematics, we assume that our systems use the function based language of higher order arithmetic, rather than the set based language.  Accordingly, $2^\setN$ is used throughout this paper to denote the set of all functions from $\setN$ to $\{0,1\}$.  

Many of our results will use fragments of the quantifier-free choice scheme. For types $\rho$ and $\tau$, we have the scheme
\[
\QF\AC^{\rho,\tau}\colon (\forall x^\rho)(\exists y^\tau) A(x,y) \to (\exists Y^{\rho \to \tau})(\forall x^\rho) A(x, Y(x)),
\]
where $A$ is a quantifier free formula. Here $A$ can have parameters of arbitrary type. 

The system $\RCAo^\omega = \mathsf{E}\text{-}\mathsf{PRA}^\omega + \QF\AC^{1,0}$ is a fragment of higher order arithmetic. It is axiomatized by a set of basic axioms along with induction for $\Sigma^0_1$ formulas and the choice scheme $\QF\AC^{1,0}$. The syntax has term-forming operations for $\lambda$ abstraction and primitive recursion.

The system $\RCAo^2$ is a second order fragment of $\RCAo^\omega$,
with only types $0$ and $1$ for elements of $\setN$ and functions $\setN \to \setN$, respectively. 
Formally, we have $\RCAo^2 = \mathsf{E}\text{-}\mathsf{PRA}^\omega + \QF\AC^{0,0}$.  
This system is equivalent to the set based system $\RCAo$ presented by Simpson~\cite{sim09}, and we will henceforth denote $\RCAo^2$ by $\RCAo$ when no confusion is likely. 

A sequence $\langle f_n : n \in \setN\rangle$ is viewed as a map $f\colon \setN \times \setN \to \setN$, so that $f_n(m) = f(\langle n,m\rangle)$, where $\langle \cdot,\cdot\rangle$ is a suitable pairing function.

Emulating Kohlenbach~\cite{koh05}, we use parentheses around the name of a functional to denote the principle stating that the functional exists. For example, the principle $\prince$ asserts the existence of the functional $\exists^2$, defined below.

There are several ways to extend the comprehension axioms of second order arithmetic to the higher order setting. One particular functional (set) existence axiom for higher order arithmetic is $(\exists^2)$, defined by
\[
(\exists^2)\colon (\exists \phi^{1 \to 0})(\forall f)
( \phi(f) = 0 \leftrightarrow (\exists n) [ f(n) = 0] ).
\]
The functional $\phi^{1\to 0}$ from this principle is itself called $\exists^2$.
The system 
$\ACAo^\omega \equiv \RCAo^\omega + \prince$ implies the arithmetical comprehension scheme. Kohlenbach~\cite{koh05} showed that $\ACAo^\omega$ is conservative over $\ACAo$ for sentences in the language~$L_2$.

Other functional existence principles correspond to $\ACAo$.  Kohlenbach~\cite{koh05} presents two such functionals. One, $\mu_0$, selects a zero of a function if such a zero exists:
\[
\princmuo\colon  (\exists \mu_0^2 )(\forall f^1)
[(\exists n^0) (f(n)=_0 0) \to        
f(\mu_0 (f))=0].
\]
Another returns the least zero of a function, in the fashion of Feferman~\cite{fefhbk}*{\S2.3.3}:
\[
  \princmu\colon (\exists \mu^2)( \forall f^1) (
   (\exists n^0) (f(n)=_0 0) \to            
  [f(\mu(f)) = 0 \land (\forall t< \mu (f) ) (f(t) \neq 0 )]).
  \]

\begin{proposition}\label{K3.9}
The following are pairwise equivalent over $\rcaw$: $\prince$, $\princmuo$, and $\princmu$.
\end{proposition}

\begin{proof}
  Kohlenbach~\cite{koh05}*{Proposition~3.9}
  proves 
  the equivalence of $\prince$ and $\princmuo$.
  Because any $\mu$ satisfying $\princmu$ also satisfies $\princmuo$,
  it suffices to show that
  $\rcaw$ proves thaxt $\princmuo$ implies $\princmu$.   Given a functional $\mu_0$ as in the
  definition of $\princmuo$,
$\mu (f)$ is the least $t\le \mu_0 (f)$ such that $f(t) = 0$.  This functional is primitive recursive in $\mu_0$ and thus exists by $\rcaw$ and $\princmuo$.
\end{proof}


\section{Countability}\label{sec:count}

One motivation for this research is the question: how difficult is it to prove $2^\setN$ is uncountable? As usual, being uncountable simply means not being countable. There are many ways to express the principle that $2^\setN$ is countable, with the following three being particularly natural:

\begin{itemize}
\item $\Cenum$: there is a sequence $\langle f_n : n \in \setN\rangle$ such that for all $g\in 2^\setN$ there is an $n \in \setN$ with $g = f_n$.

\item $\Cinj$: there is a functional $\Phi^{1\to 0}$ that is an injection from 
$2^\setN$ to $\setN$.

\item $\Cbij$: there is a functional $\Phi^{1\to 0}$ that is an bijection from 
$2^\setN$ to $\setN$.
\end{itemize}

The principles $\Cinj$ and $\Cbij$ cannot be stated in the language of second order arithmetic, but they can be stated in~$\rcaw$. When we say  we assume $\Cinj$ or $\Cbij$, this means we assume the existence of a functional with the property stated. Similarly, if we assume $\lnot \Cinj$ or $\lnot \Cbij$, this means we assume no functional has the property stated.

In context of set theory there is little reason to distinguish between $\Cinj$ and $\Cbij$, because of the comprehension principles available. As discussed below, there are key distinctions between these principles in the context of theories of arithmetic with restricted comprehension principles.

Of course, $\Cenum$, $\Cbij$, and $\Cinj$ are classically false. There are two key questions:
which systems are ``strong enough'' to disprove  these false principles, and which are ``weak enough'' to be consistent with one or more of the principles. As is well known, Cantor's diagonalization proof allows us to disprove $\Cenum$ in very weak systems (compare Theorem II.4.9 of Simpson~\cite{sim09} showing $\mathbb{R}$ is uncountable in $\RCAo$).

\begin{proposition} $\RCAo^2$ proves $\lnot \Cenum$.
\end{proposition}

\begin{proof}
Given $\langle f_n : n \in \setN\rangle$ witnessing $\Cenum$, $\RCAo^2$ can construct the function 
$g$ defined by $g(m) =1 - f_m(m)$.
Then $g \in 2^\setN$, but $g$ cannot be $f_k$ for any $k\in \setN$. 
\end{proof}

The principles $\Cinj$ and $\Cbij$ have much more interesting behavior. Normann and Sanders~\cite{ns-2020} provide a detailed analysis of the negations of these principles, which they name $\textsf{NIN}$ and $\textsf{NBI}$, respectively. (They formulate $\textsf{NIN}$ and $\textsf{NBI}$ for $\mathbb{R}$ but the results hold equally for $2^\setN$.) Their Theorem~3.2 shows that the true principle $\textsf{NIN}$ is not provable in the system $\mathsf{Z}^\omega_2 + \QF\AC^{0,1}$ (which includes $\Pi^1_\infty$ comprehension with parameters of type~$1$), and hence this system is consistent with $\Cinj$~\cite{ns-2020}*{Theorem~3.26}.  
 %
They also show that $\textsf{NIN}$ is provable in $\mathsf{Z}^\Omega_2+ \QF\AC^{0,1}$, which includes the functional $\exists^3$ in addition to $\Pi^1_\infty$ comprehension.  In the remainder of this section, we discuss some aspects of their results related to $\Cinj$ and $\Cbij$.

A key issue in analyzing $\Cinj$ is that the range of an injection from $\setN^\setN$ to $\setN$ may be hard to form with weak comprehension axioms. We will see that a similar issue arises in the study of Banach's theorem, as well, where the existence of the range of a functional becomes a key question.   By contrast, it is relatively easy to disprove~$\Cbij$~\cite{ns-2020}*{Theorem~3.28}.

\begin{proposition}[$\rcaw + \QF\AC^{0,1}$]\label{p:norange}
There is no injection $\Phi\colon 2^\setN \to \setN$ for which the characteristic function for the range exists.
In particular, $\Cbij$ is disprovable in $\rcaw + \QF\AC^{0,1}$.
\end{proposition}
\begin{proof}
  We will work in $\rcaw + \QF\AC^{0,1}$ and assume there is a bijection $\Phi$ from $2^\setN$ to $\setN$ with range $D = \{ n : (\exists g) [\Phi(g) = n]\}$ given by a characteristic function.  We will prove the principle~$\Cenum$ by constructing a kind of left inverse of~$\Phi$, which will be a (possibly noninjective) enumeration of $2^\setN$.  Because $\RCAo$ proves $\lnot \Cenum$, this gives a contradiction. 

  By assumption, for each $n \in D$ there is a $g \in 2^\setN$ with $\Phi(g) = n$.
Therefore, by $\QF\AC^{0,1}$, we may form a function $f$ so that 
$(\forall n)[n \in D \to \Phi(f_n) = n]$.
Then $\langle f_n : n \in \setN\rangle$ is an enumeration of $2^\setN$, so $\Cenum$ holds,
a contradiction.
\end{proof}

We now explain how the lemma implies certain higher order formulations of the Schroeder--Bernstein theorem
are nontrivial. Suppose, in the context of set theory, we wanted to try to use the Schroeder--Bernstein theorem to show $2^\setN$ is countable. 
Because there is a trivial injection from $\setN$ to $2^\setN$, the other
assumption in the Schroeder--Bernstein theorem is the existence of an injection from $2^\setN$ to $\setN$, that is, $\Cinj$. The conclusion is the existence of a bijection, that is, $\Cbij$. We can thus
view the implication $\Cinj \to \Cbij$ as a specific formal instance of the Schroeder--Bernstein theorem. (Normann and Sanders~\cite{ns-2022} study a different formulation of the Schroeder--Bernstein theorem, which they call $\mathsf{CB}\setN$.) 

\begin{corollary} The implication $\Cinj \to \Cbij$ is not provable in $\mathsf{Z}^\omega_2 + \QF\AC^{0,1}$. 
\end{corollary}
\begin{proof}
$\mathsf{Z}^\omega_2 + \QF\AC^{0,1}$ is consistent with $\Cinj$ but not $\Cbij$. 
\end{proof}

Lemma~\ref{p:norange} can also be used to obtain an upper bound on the strength required
to disprove~$\Cinj$. Normann and Sanders prove a version of the following lemma using
the principle $(\exists^3)$ as a formalization of
$\Sigma^1_1$ comprehension with functional parameters.  

\begin{corollary}[see Normann and Sanders~\cite{ns-2020}*{Theorem~3.1}]
  $\Cinj$ is disprovable from $\RCAo + \QF\AC^{0,1}$ along with $\Sigma^1_1$ comprehension
  with parameters of type~$2$.  
\end{corollary}
\begin{proof}
  Assume $\Phi$ is a functional witnessing $\Cinj$.
  Applying $\Sigma^1_1$ comprehension with parameter $\Phi$, we can construct the range of~$\Phi$. We then obtain a contradiction from Proposition~\ref{p:norange}.
\end{proof}

A final point of interest is that the classically false principle $\Cinj$, although consistent with $\rcaw$, has nontrivial set existence strength. Normann and Sanders discuss the contrapositive of the following proposition in the guise of a ``trick'' related to excluded middle~\cite{ns-2020}*{\S3}.

\begin{proposition}\label{prop:injaca}
$\Cinj$ implies $(\exists^2)$ over $\RCAo^\omega$.
\end{proposition}
\begin{proof}
If $\Phi$ is an injection from $2^\setN$ to $\setN$ then $\Phi$ is discontinuous at every point (here we identify elements of $\setN$ with constant functions from $\setN$ to~$\setN$). The existence of a discontinuous functional implies $(\exists^2)$ by Proposition~3.7 of Kohlenbach~\cite{koh05}.
\end{proof}

Thus, for example, there is no model of $\rcaw$ in which $\Cinj$ holds and every
element of $2^\setN$ is computable.   

\section{Bounding calculations of type 1 functions}\label{sec:bounding}

This section contains several technical lemmas related to the range of a function
$f \colon \setN \to \setN$.  Each function of this type has a number of auxiliary functions related to its range.  The most obvious is the characteristic function for the range.  We write $\rho (f,g)$ as shorthand for the formula asserting that $g$ is the characteristic function for the range of $f$.  More formally,
\[
\rho(f,g) \text{~is~} (\forall n)[(\exists m)(f(m)=n) \leftrightarrow g(n)>0 ].
\]
A bounding function can also be used to compute the range of $f$.  We write
$\beta (f,g)$ for the formula asserting that $g$ is such a bounding function.  Formally,
\[
\beta (f,g) \text{~is~}
(\forall n )[(\exists m)(f(m)=n) \leftrightarrow (\exists t \le g(n))(f(t)=n)].
\]

The results below address the problem of converting between the characteristic function
for the range of a function and a bounding function for the range, and the amount of uniformity present in the conversion.  In the second order setting, principles asserting the existence of characteristic functions and the existence of bounding functions are interchangeable, as shown by the following two results.  

\begin{proposition}[$\RCAo$]\label{prop:trans1}
 For all $f\colon \setN \to \setN$, ($\exists g)\, \rho(f,g) \leftrightarrow (\exists h)\, \beta(f,h)$.
\end{proposition}
 \begin{proof}
   Working in $\RCAo$, suppose $g$ is a characteristic function for the range of a function~$f$.  Then
   $h (n) = ( \mu\, t)  (g(n)=0 \lor f(t) = n)$ is the desired bounding function and exists by
 recursive comprehension.
 
 Now suppose that $h$ is a bounding function for the range of~$f$.  The
 characteristic function $g\colon \setN \to \setN$ can be defined by the formula
 \[
 g(n) =
 \begin{cases}
 1, & {\text {if~}}(\exists t\le h (n) ) (f(t)=n),\\
 0, & {\text {otherwise}},
 \end{cases}
 \]
and hence $g$  exists by recursive comprehension.
 \end{proof}
 
 The relationship between characteristic and bounding functions is uniform in the
 sense that $\RCAo$ proves the sequential extension of the previous result.
 
 \begin{proposition}[$\RCAo$]
   For every sequence $\langle f_i \rangle _{i \in \setN }$ of functions
      from $\setN$ to~$\setN$, we have
   \[
   (\exists \langle g_i \rangle _ {i \in \setN})( \forall n)\, \rho (f_n , g_n )
 \leftrightarrow
 (\exists \langle h_i \rangle _ {i \in \setN})(\forall n)\, \beta (f_n , h_n ).
 \]
 \end{proposition}
 
 \begin{proof}
We will write 
   $\langle f_i \rangle _{i \in \setN }$ as a function of two variables,
 so $f(i,n) = f_i (n)$.  Adapting the proof of the preceding result, write
 \[
 h (i,n) = (\mu\, t)( g(i,n) = 0 \lor f(i,t) = n)
 \]
 and
 \[
 g(i,n) =
  \begin{cases}
 1, & {\text {if~}}(\exists t\le h (i,n)) (f(i,t)=n),\\
 0, & {\text {otherwise}},
 \end{cases}
 \]
 to translate the sequences of auxiliary functions in $\RCAo$.
 \end{proof}
 
 Third order arithmetic can formalize ``translating functionals'' of type~2 to convert between characteristic functions
 and bounding functions for ranges.  Principles asserting the existence of the translating functionals provide additional
 examples of Skolemized uniformity, distinct from the sequential uniformity often considered in second order settings.

As shown below, the existence of a translating functional from bounding
 functions to characteristic functions can be proved in $\rcaw$.  However, the
 reverse translation functional requires stronger assumptions.  Thus the
 interchangeability of the two sorts of auxiliary functions witnessed in
 the second order setting by the previous two propositions does not extend to Skolemized functional formulations in
 third order arithmetic.
 
 \begin{proposition}[$\rcaw $]  There is a functional $T_{\beta \to \rho}$ of type $1\to1$ that translates bounding functions into characteristic functions for ranges.  That is, 
 \[
 (\exists T_{\beta \to \rho} )(\forall f)(\forall g) [ \beta(f,g) \to \rho (f, T_{\beta \to \rho} (f,g))].
 \]
 \end{proposition}
 
 \begin{proof}
 Working in $\rcaw$, by $\QF\AC^{1,0}$ there is a functional
 $Y$ from  $\setN^{<\setN}\times \setN^{<\setN}\times \setN$ to $\{0,1\}$
 such that
 $Y(f,g,n)=1$ if and only if $(\exists t \le g(n))[f(t) = n]$.
 Note that the defining formula is quantifier free because the
 bounded quantifier can be rewritten using a primitive recursive functional. 
 The desired functional $T_{\beta\to\rho} (f,g)$
is then 
 \[
 T_{\beta\to\rho}(f,g) = \lambda n.Y(f,g,n).\qedhere
 \]
 \end{proof}
 
 \begin{proposition}[$\rcaw$]  The following are equivalent:
 \begin{enumerate}
 \item $\prince$.\label{K4.4.1}
 \item There is a functional $T_{\rho\to \beta}$ of type $1\to1$ that translates
 characteristic functions for ranges into bounding functions.  That is:\label{K4.4.2}
 \[
 (\forall f)( \forall g )[\rho (f,g) \to \beta (f, T_{\rho \to \beta} (f,g) )].
 \]
 \end{enumerate}
 \end{proposition}
 
 \begin{proof}
   To prove that (\ref{K4.4.1}) implies~(\ref{K4.4.2}), assume $\rcaw + \prince$.  The base system $\rcaw$ suffices to prove the existence of the functional $\chi$ which takes the (type $1$ code for the) pair $(f,n)$ and maps it to the function $f_n\colon \setN \to 2$ satisfying $f_n(t) = 0 \leftrightarrow f(t)=n$.  By Proposition~3.9 of Kohlenbach~\cite{koh05}, 
   $\prince$ implies the existence of Feferman's $\mu$ functional satisfying the formula
 \[
 (\forall f) [(\exists x) (f(x)=0) \to f(\mu (f)) = 0 ].
 \]
 The function $T_{\beta\to \rho} (f,g) = \mu (\chi (f,n))$ is the desired bounding function.
 
In the proof of the preceding implication, the principle $\prince$ is sufficiently strong that we can discard the given characteristic function and calculate the bounding function directly from $f$ and~$\exists_2$. We now show we can calculate
$\exists_2$ from any translation functional $T_{\rho \to \beta}$ in the reverse direction.

To prove that (\ref{K4.4.2}) implies (\ref{K4.4.1}), suppose $T_{\rho \to \beta}$ is a translation functional as described in (2). We will show that this functional is not $\varepsilon\text{-}\delta$ continuous in the sense of Definition~3.5 of Kohlenbach~\cite{koh05}.

We can view inputs $f$ and $g$ as a single sequence
$\langle f(0), g(0), f(1), g(1) \dots\rangle$ and use the usual Baire space topology.
The functional $T_{\rho \to \beta}$ is defined for every input of two type 1 arguments, including inputs $f$ and $g$ for which $g$ is not a a characteristic function for the range of $f$.  For example, let $f_1\colon \setN \to \setN$ satisfy $f_1 (n) = 1$ for all $n$.  Let $g_2\colon  \setN \to \setN$ satisfy $g_2 (0) = g_2 ( 1) = 1$ and $g_2(n)=0$ otherwise.
Then $g_2$ is not a correct characteristic function for $f_1$.  However, $T_{\rho\to\beta}(f_1 , g_2 ) = h$ for some totally defined type 1 function $h$, and $h(0)=b$ for some value $b$.

Suppose by way of contradiction that $T_{\rho \to \beta}$ is $\varepsilon\text{-}\delta$ continuous.  Then for every pair $(f,g)$ in some neighborhood $N$ of $(f_1 , g_2 )$ we must have $T_{\rho \to \beta}(f,g) (0) = b$.  Let $f_2\colon \setN \to 2$ be a function that is $1$ for every $t\le b$, outputs a sufficient number of ones so that $(f_2 , g_2)$ is in the neighborhood $N$, and is eventually constantly zero.  Then $g_2$ is a correct characteristic function for $f_2$, so $\rho(f_2 , g_2)$ holds. However, $T_{\rho \to \beta} (f_2 , g_2 ) (0) = b$.  Thus $T_{\rho\to\beta}(f_2 ,g_2 )$ is not a bounding function for $f_2$ because $(\exists t)[ f_2(t)=0]$ but $(\forall t \le T_{\rho \to \beta}(f_2 ,g_2) (0))[f_2 (t) \neq 0]$.  Thus $\beta(f_2,T_{\rho \to \beta}(f_2 ,g_2 ))$ fails.  This contradicts the implication given in item (\ref{K4.4.2}) of the proposition. Thus $T_{\rho\to \beta}$ must not be $\varepsilon\text{-}\delta$ continuous.  By Proposition~3.7 of Kohlenbach~\cite{koh05}, $\prince$ follows.
 \end{proof}
 
 Let $R$ be the Weihrauch problem taking a type 1 function as an input and yielding output consisting of the characteristic function of the range of the input.
 Let $B$ be the Weihrauch problem that outputs  bounding functions as described above.
 Ideas from the proof of Proposition \ref{prop:trans1} can be adapted to show that $R$ and $B$ are weakly Weihrauch equivalent, and strongly Weihrauch incomparable.
 Summarizing, analyses based on sequential second order statements, Skolemized higher order statements, and Weihrauch reducibility
 yield different results.  This indicates that there are three distinct notions of uniformity considered here.
 
\section{Realizers for omniscience principles}\label{sec:realize}
 
The principle $\prince$ is closely related to a certain formulation of the limited principle of omniscience. The Weihrauch problem $\lpo$ asks for a realizer that determines whether an infinite sequence of natural numbers contains a zero. Indeed, the definition of $\prince$ could be rewritten as
\[
\prince\colon (\exists R_\lpo)(\forall f^1)[R_\lpo (f)=0 \leftrightarrow (\exists n) (f(n)=0)]
\]
to emphasize $\prince$ asserts the existence of a realizer for this problem,

The Weihrauch problem $\llpo$, related to the lesser limited principle of omniscience, asks for a realizer to identify a parity (even or odd) on which a sequence of numbers is zero, assuming either that all even positions are zero or all odd positions are zero.
We will use a principle asserting the existence of a realizer for~$\llpo$:
\begin{align*}
  (\llpo)\colon (\exists R_\llpo \leq 1)(\forall f^1)
  (  [ (\forall n)(f(2n)&=0) \lor (\forall n)(f(2n+1)=0)] \\
&  \to (\forall n)[f(2n + R_{\llpo}(f)) = 0]).
\end{align*}
Often it is more convenient to work with an equivalent form that asks for the parity of the first location where a sequence is zero, if there is such a location:
\begin{align*}
(\llpomin)\colon (\exists R_{\llpomin})&(\forall f^1)(\forall n)\\
&  [f(n)=0 
    \to R_{\llpomin} (f) \equiv_{\text{mod~} 2} (\mu\, t \le n)
    ( f(t) =0 )].
\end{align*}

For example, suppose $f= \langle 1,0,1,0,0 \dots \rangle$
denotes the infinite sequence consisting of $1, 0, 1$ followed by all zeros.
Then $R_\lpo (f) = 1$ because the sequence contains a $0$; 
$R_{\llpo} (f) = 1$ because $f(2n+1) = 0$ for all $n$; and $R_{\llpomin}(f) = 1$ because the first zero occurs in position $1$, which is odd.
For the sequence $g = \langle 1,1,0,0,\ldots\rangle$,
$R_\lpo(g) = 1$; $R_{\llpomin} = 0$ because the first zero occurs at position $2$, which is even; and the value of $R_{\llpo}$ is not determined by its defining formula.

One motivation of $\llpomin$ is that its value is determined for every sequence that includes a zero. The next proposition shows that $\llpo$ and $\llpomin$ are equivalent for our purposes.
For Weihrauch problems $\mathsf{P}$ and $\mathsf{Q}$ expressible in the language of
$\rcaw$, we say that $\rcaw$ proves $\mathsf{P} \leq_{\text{sW}} \mathsf{Q}$ if $\rcaw$
  proves there are functionals $\phi, \psi\colon \setN^\setN \to \setN^\setN$ such that, for every realizer $R_{\mathsf{Q}}$ of $\mathsf{Q}$, the functional 
  $\psi \circ R_{\mathsf{Q}} \circ \phi$ is a realizer of~$\mathsf{P}$.
  

\begin{proposition}\label{p:llpoequiv}
  $\rcaw$ proves that the problems $\llpo$ and $\llpomin$ are strongly Weihrauch equivalent, and
  that the principles $(\llpo )$ and $( \llpomin )$ are equivalent.
\end{proposition}

\begin{proof}
  First, assume $R$ is a realizer for $\llpomin$. Define a preprocessing
  function $h\colon \setN \to \setN$ such that $h(n) = 0$ if $n \not = 0$
  and $h(0) = 1$. Define a postprocessing function $w(n) = 1 - (n \operatorname{mod} 2)$.
  Then
  $
  S = w \circ R \circ h
  $
  is a realizer for $\llpo$.
  
  To see this, assume $g$ is an instance of $\llpo$. If $g$ is identically zero,
  then whichever value in $\{0,1\}$ is produced by~$S$ is acceptable. If
  $g$ is not identically zero, then $h\circ g$ is zero on exactly the inputs
  where $g$ is nonzero. Thus $R(h \circ g)$ is the parity of the first location
  where $g$ is nonzero, and $w\circ R(h \circ g)$ is the parity for which $g$ is always zero.
  This shows $\llpo \leq_{\text{sW}} \llpomin$.
  
  Conversely, suppose $S$ is a realizer for $\llpo$.  Define a preprocessing function
  $J \colon \setN^\setN \to \setN^\setN$ such that, for $f \in \setN^\setN$,
  \[
  J(f)(n) = \begin{cases}
    1, & \text{if } (\exists m < n)(f(m) = 0),\\
    f(n), & \text{otherwise}.
  \end{cases}
  \]
  Thus $J(f)$ and $f$ agree through the first zero of $f$, but afterwards $J(f)$ takes
  only the value $1$.  Hence there is at most one input $k$
  for which $h(J(f))(k)$ is nonzero, and if there is such a $k$ then it is the least
  input for which $f(k) = 0$.  This means that $h(J(f))$ is in the domain of $\llpo$, and 
  $S(h(J(f)))$ produces the parity of $k$. Thus $S \circ (h \circ J)$ is a realizer for~$\llpomin$. We have shown $\llpomin \leq_{\text{sW}} \llpo$.

  The preprocessing and postprocessing functionals in this argument can all be formed in $\rcaw$, which can verify the correctness of the argument. None of the postprocessing functions require access to the original instance of a problem. Hence $\rcaw$ proves
that  $\llpo$ or $\llpomin$ are strongly Weihrauch equivalent.

Concatenation of the preprocessing and postprocessing functionals with any $\llpo$ realizer yields an $\llpomin$ realizer, so
the principle $(\llpo )$ implies the principle $( \llpomin)$.  The converse follows in a similar fashion.
\end{proof}


Results of Weihrauch analysis include $\llpo <_W \lpo$ and
the parallelized form $\widehat{\mathsf{LLPO}} <_W \widehat{\mathsf{LPO}}$.
See Weihrauch~\cite{W-1992}*{\S4} and
Brattka and Gherardi~\cite{bg-2011}*{Theorem~7.13} for proofs.  Consequently, the following
result may intially be surprising.   The underlying difference is that
Weihrauch reducibility requires a single reduction that works for all realizers;
the argument below breaks into cases depending on the behavior of the realizer.

 \begin{proposition}[$\rcaw$]\label{llpoimpliesE2}
    $(\llpo )$ implies $\prince$.
\end{proposition}

\begin{proof}
  Working in $\rcaw$, 
  by Proposition \ref{p:llpoequiv},
  it is sufficient to assume the existence of $R = R_\llpomin$ as provided by $(\llpomin)$, and prove $\prince$ holds.

  Let $f= \langle 1,1,1\dots \rangle$ be the infinite
sequence of ones.  Our goal is to show that $R$ is sequentially discontinuous at $f$.  We will construct
a sequence $\langle g_n \rangle$ such that $\lim_{n \to \infty} g_n = f$ and for each $n$,
$R (g_n)$ disagrees with $R (f)$.  In particular, if $R(f) = 1$, we want
$R (g_n ) = 0$ for all $n$, so we define $g_n$ as a sequence of $2 + 2n$ ones followed by all zeros.
On the other hand, if $R (g_n ) = 1$, we want $R(g_n ) = 1$ for all $n$, so we define
$g_n$ as a sequence of $1+2n$ ones followed by all zeros.  Summarizing, for each $n$ and $m$ we have
\[
g_n(m) = 
\begin{cases}
1,&{\text{if~}}m<1+R(f) + 2n,\\
0,&\text{otherwise}.
\end{cases}
\]
Note that $\rcaw$ proves the existence of the sequence $\langle g_n \rangle$, that $\lim_{n \to \infty} g_n = f$,
and that for all $n$, $R (g_n ) \neq R (f)$.  Thus $R$ is sequentially discontinuous and $\prince$ follows by Proposition~3.7 of Kohlenbach~\cite{koh05} (see Proposition~\ref{K3.7}).

The proof of Kohlenbach's proposition is based on the proof of Lemma~1 of Grilliot~\cite{grilliot}. We append that argument here to give a direct derivation of $\prince$ from $(\llpomin)$.  Let the function $f$ and the sequence $\langle g_n \rangle$ be defined as above.  $\rcaw$ suffices to prove the existence of the operator $J \colon \setN^\setN \to 2^\setN$ defined for $h \colon \setN \to \setN$ and $j\in \setN$ by
\[
J(h)(j)= 
\begin{cases}
1, &\text{if ~}(\forall x \le j) [h(x) \neq 0 ],\\
g_i, &\text{if ~} i \le j \land i = (\mu t)[h(t) =0].
\end{cases}
\]
Note that $J(h) = f$ if $(\forall x)[h(x) \neq 0]$.  On the other hand, if $i$ is the least value for which $h(i)=0$,
then  $J(h) = g_i$.  Consequently, for all $h \colon \setN \to \setN$, $R (J(h)) = R (f)$ if and only if
$(\forall x)[h(x) \neq 0]$.  Thus
$R_\lpo (h) = 1 - | R (J(h))-R (f)|$, so the existence of $R_\lpo$ follows by $\rcaw$.
\end{proof}

\begin{theorem}[$\rcaw$]\label{E2LLPO}  The following are equivalent:
\begin{enumerate}
\item  $(\exists ^2 )$.
\item  $(\llpo )$.
\end{enumerate}
\end{theorem}

\begin{proof}
To show that (1) implies (2), as noted in Proposition \ref{K3.9}, Proposition~3.9 of Kohlenbach~\cite{koh05} shows that the principle $( \exists ^2 )$ proves the
existence of Feferman's $\mu$ functional which satisfies:
\[
(\forall f)[ (\exists x) (f(x) = 0) \to [f(\mu (f))=0 \land (\forall t< \mu(f)) (f(t) \neq 0 )]]
\]
The remainder function ${\sf rm}(n,2)$ yielding the remainder of dividing $n$ by $2$ is primitive recursive.  Thus $\rcaw$ proves the existence of the composition functional
${\sf rm}(\mu (f),2)$, which satisfies the definition of $( \llpo )$.

The converse was proved as Proposition~\ref{llpoimpliesE2} above.
\end{proof}

For an alternative proof of the forward implication of Theorem~\ref{E2LLPO}, we can use
a formalized Weihrauch reducibility result.  The next two results illustrate this process.
The following proposition converts formal Weihrauch reducibility to proofs of implications
of Skolemized functional existence principles.

\begin{proposition}\label{FWRimp}
If $\sf P$ and $\sf Q$ are problems and $(\sf P)$ and $(\sf Q)$ are the associated Skolemized functional existence principles,
then $\rcaw$ proves \[{\sf} P \wlt {\sf Q} \to (({\sf Q})\to ({\sf P})).\]
\end{proposition}

\begin{proof}
Working in $\rcaw$, suppose ${\sf} P \wlt {\sf Q}$ and $({\sf Q})$ hold.  Let $\varphi$ and $\psi$ be the functionals
witnessing ${\sf} P \wlt {\sf Q}$ and let $R_{\sf Q}$ witness $({\sf Q})$.  $\rcaw$ proves the existence of the composition
$\psi (R_{\sf Q} (\varphi (x)),x)$, which can be directly shown to realize the principle $({\sf P})$.
%
\end{proof}

Next, we prove the formalized Weihrauch reducibility result corresponding to the forward implication
of Theorem~\ref{E2LLPO}.

\begin{proposition}\label{FWRllpo}
  $\rcaw$ proves $\llpo \leq_{\text{sW}} \lpo$.
\end{proposition}

\begin{proof}
  We again work with $\llpomin$ in place of~$\llpo$.
Let $\varphi \colon \setN^\setN \to 2^\setN$ be the preprocessing functional defined by:
\[
\varphi (h)(n) =
\begin{cases}
1,& \text{if~}(\forall t \le n) (h(t)>0),\\
1,& \text{if~}(\exists t \le n) (h(t) = 0)\text{~and the least such~}t\text{~is odd,}\\
0,& \text{if~}(\exists t \le n) (h(t) = 0)\text{~and the least such~}t\text{~is even.}
\end{cases}
\]
Note that $\varphi (h)$ is the sequence of all ones except when the first zero in the
range of $h$ occurs in an even location.  Define the postprocessing functional $\psi (h,n) = n$. If $R_\lpo$ is a realizer for $\lpo$, then $\psi (h, R_\lpo (\varphi(h))$ is a realizer for $\llpo$.
Because $\psi$ makes no use of $h$, this shows that $\llpomin \leq_{\text{sW}} \lpo$. Applying Proposition~\ref{p:llpoequiv}, we see that $\llpo  \leq_{\text{sW}} \lpo$.
\end{proof}

As mentioned before, the forward implication of Theorem~\ref{E2LLPO} follows immediately from Proposition~\ref{FWRllpo} and Proposition~\ref{FWRimp}.  The next result uses Theorem~\ref{E2LLPO} to give a short proof of one direction of Proposition~3.4 of Kohlenbach~\cite{kohWKL}, showing that a uniform version of weak K\"onig's lemma is equivalent to $(\exists ^2 )$.  This equivalence is also included in Proposition~3.9 of Kohlenbach~\cite{koh05}.   This equivalence will be helpful in the analysis of Banach's Theorem for $\setN$ in the next section.
(For a discussion of uniform $\sf{WWKL}$ see Theorem 3.2 of Sakamoto and Yamazaki~\cite{sy}.)

\begin{proposition}[$\rcaw$]\label{E2WKL}
 The following are equivalent:
\begin{enumerate}
\item  $(\exists ^2 )$.
\item $( \WKL )$  There is a functional $\WKL\colon \setN^\setN \to 2^\setN$ such that
 if $T$ is a code for an infinite tree in $2^\setN$, then $\WKL (T)$ is an infinite path
 in $T$.
\end{enumerate}
\end{proposition}

\begin{proof}
As noted in the proof of Proposition~3.4 of Kohlenbach~\cite{kohWKL}, the proof that (1) implies (2) follows from the fact that, given the
functional $\exists ^2$, primitive recursion can define a functional which selects an infinite branch of an infinite binary tree.
For a short proof of the converse, it suffices to show that $(\WKL )$ implies $(\llpomin)$.  Consider an instance
$f\colon \setN \to 2$ for $\llpomin$.  Let $\langle 1 \rangle_n$ denote the sequence of $n$ ones.  Define the
$0\mhyphen 1$ tree $T_f$ by:
\begin{itemize}
\item Only sequences of the form $0 \cat \langle 1 \rangle _n$ and $1\cat \langle 1 \rangle _n$ are in $T_f$,
\item $0 \cat \langle 1 \rangle _n \in T_f$ if and only if either $(\forall t\le n)[f(t) \neq 0]$ or $(\mu t \le n)[f(t) =0]$ is even, and
\item $1 \cat \langle 1 \rangle _n \in T_f$ if and only if either $(\forall t\le n)[ f(t) \neq 0]$ or $(\mu t \le n)[f(t) =0]$ is odd.
\end{itemize}
$\rcaw$ can prove the existence of a functional $\varphi$ that maps each $f$ to $T_f$.  For each $f$, the first element
of $\WKL (\varphi (f) )$ is an $\llpomin$ solution for~$f$.
\end{proof}

\section{Banach's Theorem on $\setN$}\label{sec:banachn}

This section reformulates Theorems~\ref{t:hirst1} and \ref{t:hirst2} as higher order functional existence statements. In particular, Theorem~\ref{BanachN} shows that, in the Skolemized higher order setting, the bounded version is equivalent to the unbounded version.  This collapse mimics that of the uniform principle $( \WKL )$. Our discussion begins with the formulation of the bounded principle and its proof from~$( \WKL )$.

\begin{definition}
 A bounded Banach functional $\sf{bB}_\setN$  on $\setN$ is defined as follows.  For
 injective functions $f_0\colon \setN \to \setN$ and $f_1 \colon \setN \to \setN$ with bounding functions  $b_0$ and $b_1$, ${\sf{bB}}_\setN (f_0 , f_1 , b_0, b_1 )$ is a bijective function $h\colon \setN \to \setN$  such that for all $m,n\in \setN$, $h(m)=n$ implies $f_0 (m)=n$ or $f_1(n) = m$.  As usual, the parenthesized  expression $({\sf{bB}}_\setN )$ denotes the principle asserting the existence of a bounded Banach functional for $\setN$.
 \end{definition}
 
\begin{proposition}[$\rcaw$]\label{WKLbBN}
 $( \WKL)$ implies $({\sf{bB}}_\setN )$.
\end{proposition}
 
\begin{proof}
 We work in $\rcaw$.  For bounded injections $\vec f = \langle f_0 , f_1 , b_0 , b_1 \rangle$, we will describe the computation of a related tree $T_ {\vec f}$ so that any infinite path through $T_ {\vec f}$ defines an injection
 $h$ satisfying Banach's theorem.  If $p$ is an infinite path through $T_ {\vec f}$, the bijection $h$ will be  defined by:
 \[
 h(n)=
 \begin{cases}
f_0 (n), & {\text {if~}}p(n) = 0,\\
(\mu\, t \le b_1(n))(f_1 (t) =n), & {\text {if~}} p(n)=1.
 \end{cases}
 \]
 A finite sequence $\sigma \in 2^{<\setN}$ is included in $T_{\vec f}$ if it satisfies the following four conditions, (i)--(iv), each ensuring an aspect of the back-and-forth construction of the bijection $h$.

 \newenvironment{reqlist}%
                {\begin{list}{1em}{\leftmargin=3.1em\labelwidth=1.5em}}%
                {\end{list}}
 
 First, if there is an $m< \lh(\sigma)$ which is not in the range of $f_1$, we ensure
 that $h(m) = f_0 (m)$.
 \begin{reqlist}{}{}
 \item [(i)]  If $m< \lh (\sigma)$ and $(\forall t \le b_1 (m))[f_1(t) \neq m]$ then $\sigma(m)=0$.
 \end{reqlist}
 Next, if $m$ is $f_1(t)$ for some $t$ and $t$ is not in the range of $f_0$, we set $h(m)=t$ in the following fashion.
 \begin{reqlist}
 \item [(ii)]  If $m< \lh (\sigma)$, there is a $t \le b_1(m)$ such that $f_1(t)=m$, and $(\forall s\le b_0(t))[f_0(s) \neq t]$,
 then $\sigma(m)=1$.
 \end{reqlist}
 The next clause ensures that $h$ is injective.
 \begin{reqlist}{}{}
 \item [(iii)]  If $m,n < \lh (\sigma)$, $\sigma(m) =0$, and $\sigma (n) = 1$, then $f_1(f_0 (m)) \neq n$.
 \end{reqlist}
 This final clause ensures that $h$ is surjective.
 \begin{reqlist}{}{}
 \item [(iv)] If $m,n < \lh (\sigma)$, $\sigma (m) = 0 $, and $\sigma(n) = 1$, then $f_1(f_0 (n)) \neq m$.
 \end{reqlist}
 The sequences satisfying the clauses are closed under initial segments, so $T_{\vec f}$ is a tree.
 The second order proof of the bounded Banach theorem in $\WKLo$ (Theorem~\ref{t:hirst2}) shows that $T_{\vec f}$ is infinite.  The construction of $T_{\vec f}$ terminates for arbitrary choices
 of $\vec f$, even if $f_0$ and $f_1$ are not injections or if $b_0$ or $b_1$ gives incorrect bounding
 information.  Thus $\rcaw$ proves the existence of a functional mapping $\vec f$ to $T_{\vec f}$.
 Whenever $T_{\vec f}$ is infinite, $\WKL (\vec f )$ yields an infinite path.  Concatenating these
 functionals with the one computing the bijection $h$ as described at the beginning of the proof
 yields the desired Banach functional.
 \end{proof}

The preceding proposition differs from the second order analog (Theorem~\ref{t:hirst1}) in the formulation of the bounding functions.  The original second order version was formulated with characteristic functions for the ranges of the injections.  However, in the calculation of $T_{\vec f}$, the use of an incorrect characteristic function could result in an unbounded nonterminating search, causing the functional mapping $\vec f$ to $T_{\vec f}$ to be undefined on some inputs.  This difficulty could be circumvented by using the fact that the uniform principle $(\WKL )$ implies $(\exists ^2)$, but the argument presented here uses a single application of $(\WKL )$.

Our proof of the unbounded version of Banach's Theorem for $\setN$ from $\prince$ uses a proposition relating $\prince$ to the existence of bounding functions.
 
\begin{definition}\label{BN1}
 The functional ${\sf b}\colon\setN^\setN \to \setN^\setN$ maps any function $f\colon\setN \to \setN$ to a bounding function ${\sf b}(f)$ for $f$.  In the notation of section \S5, for all $f$ we have $\beta (f , {\sf b}(f))$. As usual, the parenthesized expression
 $({\sf b})$ denotes the principle asserting the existence of a bounding functional.
\end{definition}
 
\begin{lemma}[$\rcaw$]\label{BN2}
 $\prince$ implies $({\sf b})$.
\end{lemma}
 
\begin{proof}
 $\rcaw$ proves the existence of a functional ${\sf z}\colon \setN^\setN \times \setN \to 2^\setN$ where ${\sf z} (f,n) = g$ satisfies $g(m) = 0$ if and only if $f(m)=n$.  By Proposition~\ref{K3.9}, $\prince$ proves the existence of Kohlenbach's $\mu_0$.  By composition and $\lambda$ abstraction, there is a functional ${\sf b}(f)$ mapping $f$ to the bounding function $\mu_0 ({\sf z}(f,n))$.
\end{proof}
 
The next definition formulates an unbounded form of Banach's theorem on $\setN$.  Using the principle
$( \sf b )$, the unbounded form can be derived from the bounded form.

\begin{definition}
A Banach functional on $\setN$, denoted $\sf{B}_\setN$, is defined as follows.  For
injective functions $f_0\colon \setN \to \setN$ and $f_1 \colon \setN \to \setN$,
${\sf{B}}_\setN (f_0 , f_1 )$ is a bijective function $h\colon \setN \to \setN$
such that for all $m,n\in \setN$, $h(m)=n$ implies $f_0 (m)=n$ or $f_1(n) = m$.  As usual, the parenthesized expression $({\sf{B}}_\setN )$ denotes the principle asserting the existence of a Banach functional for $\setN$.
\end{definition}
 
\begin{proposition}[$\rcaw$]\label{A10}
 $(\exists^2)$ implies $({\sf{B}}_\setN )$.
\end{proposition}
 
\begin{proof}
Working in $\rcaw$, assume $(\exists ^2 )$.  By Proposition~\ref{E2WKL}, we have $(\WKL )$, so by Proposition~\ref{WKLbBN}, we have $({\sf{bB}}_\setN )$.  By Lemma \ref{BN2}, we have the functional $\sf b$ mapping functions to associated bounding functions.  The composition functional ${\sf bB}(f_0, f_1, {\sf b}(f_0),{\sf b}(f_1))$ satisfies $({\sf{B}}_\setN )$.
\end{proof}

The next proposition essentially 
shows that the principle $(\exists^2)$ can be deduced from the restricted form of Banach's theorem for $\setN$.
 
\begin{proposition}[$\rcaw$]\label{B10}
  $({\sf{bB}}_\setN )$ implies $( \llpo )$.
\end{proposition}

\begin{proof}
Assume $\rcaw$ and $({\sf{bB}}_\setN )$. Our goal is to prove the existence of the functional $\llpomin$.  Let $g\colon\setN \to \setN$ and define bounded injections $f_0$ and $f_1$ as  follows.
  
 Figure~\ref{B10fig} illustrates the construction of $f_0$ and $f_1$ for three choices of $g$. In general, for even inputs like $n=2m$, let $f_0 (n) = n+2$ and $f_1 (n) = n$.  For $n=1$, let $f_0 (1) = 0$ and $f_1(1)=1$.  The appearance of $0$ in the range of $g$ affects the definitions of $f_0$ and $f_1$ on other odd values. 
 Suppose that $n=2m+3$.
 If $(\forall t \le n-2 )[g(t) \neq 0]$, then let $f_0(n)=n-2$ and $f_1(n)=n$.
 If $(\exists t \le n-2) [g(t) = 0]$, write $s = (\mu t)[g(t)= 0]$.  If $s$ is even, then let
 \[
 f_0(n) =
 \begin{cases}
 n-2, &\text{if}~ s=n-3,\\
 n, &\text{if}~ s<n-3,
 \end{cases}
 \]
 and let $f_1(n)=n+2$.  If $s$ is odd, then let
  \[
 f_0(n) =
 \begin{cases}
 n-2, &\text{if}~ s=n-2,\\
 n, &\text{if}~ s<n-2,
 \end{cases}
 \]
 and
  \[
 f_1(n) =
 \begin{cases}
 n, &\text{if} ~s=n-2,\\
 n+2, &\text{if}~ s<n-2.
 \end{cases}
 \]
 
 In figure, $f_0$ is represented by solid arrows and $f_1$ by dashed arrows.  Extending
 the chains to the left and right, each number has an exiting arrow, so both $f_0$ and $f_1$ are total. No number has two entering arrows, so $f_0$ and $f_1$ are injective.
 
 \begin{figure}[t]
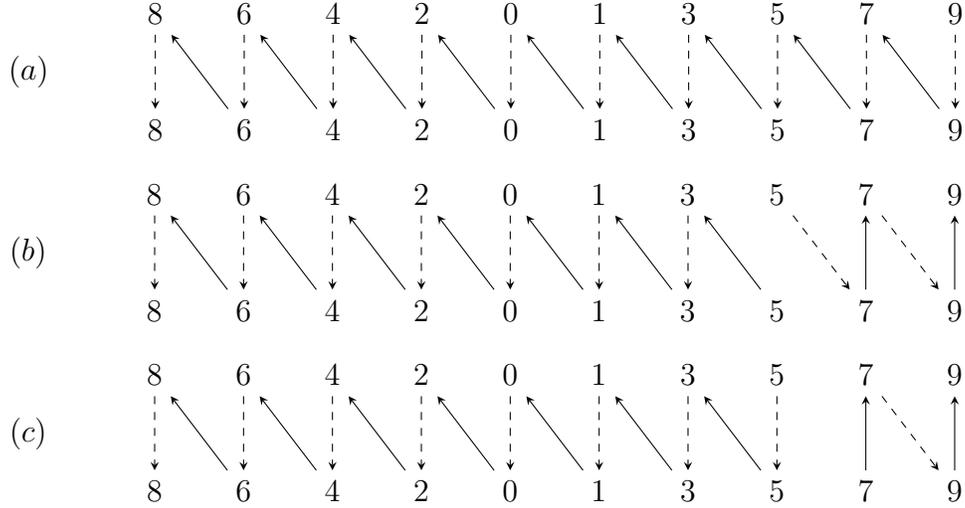

 \begin{center}
   \tikzpicture
   [mymatrix/.style={
    matrix of nodes,
    column sep=.7cm,
    row sep=1cm,
    row 2/.style={minimum height=.1cm}}
  ]

    \node at (-7,0) {$(a)$};

  \matrix[mymatrix] (m) {
    8&6&4&2&0&1&3&5&7&9 \\
    8&6&4&2&0&1&3&5&7&9 \\
  };
  \draw[dashed,-stealth](m-1-1)--(m-2-1);
  \draw[dashed,-stealth](m-1-2)--(m-2-2);
  \draw[dashed,-stealth](m-1-3)--(m-2-3);
  \draw[dashed,-stealth](m-1-4)--(m-2-4);
  \draw[dashed,-stealth](m-1-5)--(m-2-5);
  \draw[dashed,-stealth](m-1-6)--(m-2-6);
  \draw[dashed,-stealth](m-1-7)--(m-2-7);
  \draw[dashed,-stealth](m-1-8)--(m-2-8);
  \draw[dashed,-stealth](m-1-9)--(m-2-9);
  \draw[dashed,-stealth](m-1-10)--(m-2-10);
  
  \draw[-stealth] (m-2-2) -- (m-1-1);
  \draw[-stealth] (m-2-3) -- (m-1-2);
  \draw[-stealth] (m-2-4) -- (m-1-3);
  \draw[-stealth] (m-2-5) -- (m-1-4);
  \draw[-stealth] (m-2-6) -- (m-1-5);
  \draw[-stealth] (m-2-7) -- (m-1-6);
  \draw[-stealth] (m-2-8) -- (m-1-7);
  \draw[-stealth] (m-2-9) -- (m-1-8);
  \draw[-stealth] (m-2-10) -- (m-1-9);
\endtikzpicture
 
 \tikzpicture
  [mymatrix/.style={
    matrix of nodes,
    column sep=.7cm,
    row sep=1cm,
    row 2/.style={minimum height=.1cm}}
  ]

      \node at (-7,0) {$(b)$};

  \matrix[mymatrix] (m) {
    8&6&4&2&0&1&3&5&7&9 \\
    8&6&4&2&0&1&3&5&7&9 \\
  };
  \draw[dashed,-stealth](m-1-1)--(m-2-1);
  \draw[dashed,-stealth](m-1-2)--(m-2-2);
  \draw[dashed,-stealth](m-1-3)--(m-2-3);
  \draw[dashed,-stealth](m-1-4)--(m-2-4);
  \draw[dashed,-stealth](m-1-5)--(m-2-5);
  \draw[dashed,-stealth](m-1-6)--(m-2-6);
  \draw[dashed,-stealth](m-1-7)--(m-2-7);
  \draw[dashed,-stealth](m-1-8)--(m-2-9);
  \draw[dashed,-stealth](m-1-9)--(m-2-10);

  \draw[-stealth] (m-2-2) -- (m-1-1);
  \draw[-stealth] (m-2-3) -- (m-1-2);
  \draw[-stealth] (m-2-4) -- (m-1-3);
  \draw[-stealth] (m-2-5) -- (m-1-4);
  \draw[-stealth] (m-2-6) -- (m-1-5);
  \draw[-stealth] (m-2-7) -- (m-1-6);
  \draw[-stealth] (m-2-8) -- (m-1-7);
  \draw[-stealth] (m-2-9) -- (m-1-9);
  \draw[-stealth] (m-2-10) -- (m-1-10);
\endtikzpicture

 
 \tikzpicture
  [mymatrix/.style={
    matrix of nodes,
    column sep=.7cm,
    row sep=1cm,
    row 2/.style={minimum height=.1cm}}
  ]

  \node at (-7,0) {$(c)$};
  
  \matrix[mymatrix] (m) {
    8&6&4&2&0&1&3&5&7&9 \\
    8&6&4&2&0&1&3&5&7&9 \\
  };
  \draw[dashed,-stealth](m-1-1)--(m-2-1);
  \draw[dashed,-stealth](m-1-2)--(m-2-2);
  \draw[dashed,-stealth](m-1-3)--(m-2-3);
  \draw[dashed,-stealth](m-1-4)--(m-2-4);
  \draw[dashed,-stealth](m-1-5)--(m-2-5);
  \draw[dashed,-stealth](m-1-6)--(m-2-6);
  \draw[dashed,-stealth](m-1-7)--(m-2-7);
  \draw[dashed,-stealth](m-1-8)--(m-2-8);
  \draw[dashed,-stealth](m-1-9)--(m-2-10);

  \draw[-stealth] (m-2-2) -- (m-1-1);
  \draw[-stealth] (m-2-3) -- (m-1-2);
  \draw[-stealth] (m-2-4) -- (m-1-3);
  \draw[-stealth] (m-2-5) -- (m-1-4);
  \draw[-stealth] (m-2-6) -- (m-1-5);
  \draw[-stealth] (m-2-7) -- (m-1-6);
  \draw[-stealth] (m-2-8) -- (m-1-7);
  \draw[-stealth] (m-2-9) -- (m-1-9);
  \draw[-stealth] (m-2-10) -- (m-1-10);
\endtikzpicture
\end{center}
\caption{
Construction for Proposition~\ref{A10}.
  (a): $f_0$ (solid) and $f_1$ (dashed) when $0$ is not in the range of
  $g$. (b): $f_0$ (solid) and $f_1$ (dashed) when $g(2)=0$.
  (c): $f_0$ (solid) and $f_1$ (dashed) when $g(3)=0$.
}\label{B10fig}
 \end{figure}

Figure~\ref{B10fig}(a)  corresponds to the situation when $0$ does not appear in the range of $g$.
Any bijection $h$ satisfying Banach's theorem must either consist of all the (inverses of the) dashed arrows or all the solid arrows. In this situation, $h(1)$ may be $0$ or $1$.

Figure~\ref{B10fig}(b) corresponds to the case when $g(2)=0$ is the first zero in the range of $g$.  In this case, $5$ must be in the domain of $h$, so $h(5)=f_0(5)=3$. The only bijection satisfying Banach's theorem consists of solid arrows to the left of $5$, so $h(1)=0$.

Figure~\ref{B10fig}(c) is for the case when $g(3)=0$ is the first zero in the range of $g$.  Here $5$ must be in the range of $h$, so $h(5)=f^{-1}_1(5)=5$.  The only bijection satisfying Banach's theorem consists of (inverses of the) dashed arrows to the left of $5$, and so $h(1)=1$. If $0$ first appears in the range of $g$ at an even
value, the the figure for $f_0$ and $f_1$ will be a shifted version of the second figure.  Odd values yield a shifted version of the third figure.

Because $f_0(n)$ is never less than $n-2$, $b_0(n)=n+2$ is a bounding function for $f_0$.
Similarly, $f_1(n)$ is never less than $n$, so $b_1(n) = n$ is a bounding function for $f_1$.
Routine verifications show that for any choice
of $g$, $f_0$ and $f_1$ will be injections bounded by $b_0$ and $b_1$.  Suppose that $h$ is any bijection satisfying Banach's theorem for $f_0$, $f_1$, $b_0$, and $b_1$.  If the first $0$ in the range of $g$ occurs at an even value, then $h(1)=0$.  If it occurs at an odd value, then $h(1)=1$. If $0$ is not in the range of $g$, then $h(1)$ may be either $0$ or~$1$.  $\rcaw$ proves the existence of the functional mapping $g$ to the bounded injections $\vec f _g= \langle f_1 , f_1 , b_0 , b_1 \rangle$
as defined above.  The functional ${\sf{bB}}_\setN(\vec f _g )$ yields the bijection $h$ for $\vec f _g$. Consequently, the functional mapping $g$ to ${\sf{bB}}_\setN(\vec f _g )(1)$ (which equals $h(1)$)
is ${\llpomin}(g)$.
\end{proof}
 
Concatenating the preceding arguments yields the desired equivalence theorem and concludes the section.
 
\begin{theorem}[$\rcaw $]\label{BanachN}
The following are equivalent:
 \begin{enumerate}
 \item  $(\exists^2 )$.
 \item  $({\sf{B}}_\setN)$.
 \item $({\sf{bB}}_\setN)$.
 \end{enumerate}
\end{theorem}
 
\begin{proof}
 Proposition~\ref{A10} shows that (1) implies (2).  Because $({\sf{bB}}_\setN)$ is a restriction of $({\sf{B}}_\setN)$, (2) implies (3) is immediate.  Proposition~\ref{B10} and Theorem~\ref{E2LLPO} show that (3) implies~(1).
\end{proof}

\section{Banach's Theorem on compact spaces}\label{sec:banachcompact}

Our next goal is to analyze the strength of Banach's theorem restricted to uniformly continuous functions on
complete separable metric spaces.  We formalize complete separable metric spaces in the manner
of Simpson~\cite{sim09}*{II.5}.  The space $\hat A$ is the collection of rapidly converging sequences
of elements of an underlying (countable) set $A$.  The metric is a function $d\colon A\times A \to \mathbb R$,
extended to $\hat A$ by defining
$d(\langle a_i \rangle_{i \in \setN}, \langle a^\prime_i \rangle_{i \in \setN})= \langle d(a_i , a^\prime_i )\rangle_{i \in \setN}$.
As in Definition~III.2.3 of Simpson~\cite{sim09}, a space is {\sl compact} if there is an infinite sequence of finite
sequences of points of $\hat A$ of the form $\langle \langle x_{ij} : i \le n_j \rangle : j \in \setN \rangle$, such that
for all $z \in \hat A$ and $j \in \setN$, there is an $i \le n_j$ such that $d(x_{ij} , z ) < 2^{-j}$.

Uniform continuity can be witnessed by a modulus of uniform continuity as formalized in
Definition IV.2.1 of Simpson~\cite{sim09}.  The function $h\colon \setN \to \setN$ is a modulus of uniform continuity for $f$
if for all $k$, $|x-y|<2^{-h(k)}$ implies $|f(x)-f(y)|<2^{-k}$.
If $h_f$ is a modulus of uniform continuity for $f$ and $h_g$ is a modulus of uniform continuity for $g$,
then $h$ defined by $h(n) = \max \{ h_f (n) , h_g (n) \}$ is a modulus of uniform continuity for $f$ and $g$.
Consequently, a joint modulus can be used to simplify some statements.

Kohlenbach~\cite{koh05} defines two equivalent forms of continuity for functionals of type $1\to1$.  
First, $C^{1\to 1}$ is everywhere sequentially continuous if (\cite{koh05}*{Definition~3.3}):
\[
(\forall g^1)(\forall \langle g_n \rangle)[ \lim_{n\to \infty} g_n = g \to \lim_{n \to \infty} C(g_n ) = C(g)].
\]
Second, $C^{1\to 1}$ is everywhere $\varepsilon\text{-}\delta$ continuous if (\cite{koh05}*{Definition~3.5}):
\[
(\forall g^1)(\forall k)(\exists n)(\forall h^1)[d(g,h)<2^{-n} \to d(C(g),C(h))<2^{-k}].
\]
  This second definition is similar to familiar textbook definitions of continuity
for total functions.  The use of $n$ and $k$ reduces the type of the quantifiers corresponding to
$\delta$ and $\varepsilon$.  Proposition~3.6 of Kohlenbach~\cite{koh05} proves in $\rcaw$ that
$C$ is sequentially continuous if and only if $C$ is $\varepsilon\text{-}\delta$ continuous.

The following portion of Proposition~3.7 of Kohlenbach~\cite{koh05} is very useful in proving
reversals.

\begin{proposition}[\cite{koh05}*{Proposition~3.7}]\label{K3.7}
The following are equivalent over $\rcaw$:
\begin{enumerate}
\item  $\prince$.
\item  There is a functional which is not everywhere sequentially continuous.
\item  There is a functional which is not everywhere $\varepsilon\text{-}\delta$ continuous.
\end{enumerate}
\end{proposition}

For uniformly continuous functionals on compact complete separable metric spaces, it is possible
to find ranges using only $\prince$.  Indeed, the next two lemmas show that
for Cantor space the existence of ranges is equivalent,
a higher order analog of Lemma~III.1.3 of Simpson~\cite{sim09}

\begin{lemma}[$\rcaw + \prince$]\label{0628A}
 Suppose $X$ is a compact complete separable metric space.
There is functional $R$ such that if $f\colon X \to X$ is a function with modulus of uniform
continuity $h$, $R(f,h)$ is the characteristic function of the range of $f$.  That is,
for all $y\in X$, $R(f,h)(y) \in \{0,1\}$ and $R(F,h)(y)=1$ if and only if $(\exists x \in X)[f(x)=y]$.
\end{lemma}

\begin{proof}
Working in $\rcaw$, let $X$ be as hypothesized, with the compactness of $X$ witnessed
by $\langle\langle x_{ij} : i\le n_j \rangle : j \in \setN \rangle$.  Consider a function $f\colon X \to X$
with modulus of uniform continuity $h$.  Informally, a value $y \in X$ is in the range of $f$ if and
only if for every $m$ there is an $x$ with $d(F(x),y) < 2^{-m}$.  By uniform continuity
and compactness, such an $x$
exists if and only if there is an $i \le n_{h(m)}$ such that $d(f(x_{ih(m)}),y)<2^{-m}$.
In $\rcaw$, for $f\colon X \to X$ and $h\colon \setN \to \setN$, we may define
\[
K(f,h,y,m)=
\begin{cases}
1, &{\text{if~}} (\exists i \le n_{h(m)})[d(f(x_{ih(m)}),y)<2^{-m}]\\
0, &{\text{otherwise.}}
\end{cases}
\]
Viewing $K$ as a function in $m$ with parameters $f$, $h$, and $y$, in $\prince$ we may
define $R(f,h)(y) = \varphi (K(F,h,y,m))$.  Informally, by the definition of~$\varphi$,
$R(f,h)(y)=1$ if and only if for all $m$ there is an $x$ with $d(f(x),y)<2^{-m}$.  Note that
the termination of the calculation of $R(f,h)$ does not depend on the continuity of $f$
or the correctness of $h$.

To complete our proof, we must verify in $\rcaw + \prince$ our informal claim that for
each continuous function $f$ with modulus of uniform continuity $h$, $R(f,h)$ is the characteristic function
for the range of $f$.  First,
if $R(f,h)(y) = 1$, then there is a sequence $\langle x^\prime_{i_m} \rangle$ such that for every $m$,
$d(f(x^\prime_{i_m}), y) < 2^{-m}$.
The principle $\prince$ implies $\ACAo$ which implies the Bolzano--Weierstrass theorem
(see Theorem III.2.7 of Simpson~\cite{sim09}), so we can
thin $\langle x^\prime_{i_m} \rangle$ to a sequence converging to some $x \in X$.  By
sequential continuity of $f$, we have $f(x) = y$.

Second, if $R(f,h)(y) = 0$, then for some natural number $m$, we must have $(\forall i \le n_{h(m)} )[d(f(x_{i h(m)}, y)\ge 2^{-m}]$.
Suppose by way of contradiction that $f(x)=y$.  Choose $i\le n_{h(m)}$ such that $d(x, x_{i h(m)} )< 2^{-h(m)}$.
Because $f$ is uniformly continuous, $d(f(x_{i h(m)}), f(x) ) < 2^{-m}$.  Concatenating inequalities, we have
\[
2^{-m} \le d(f(x_{i h(m)} ) , y) = d(f(x_{i h(m)} ) ,f(x))< 2^{-m},
\]
a contradiction.  Thus, $R(f,h)(y) = 0$ implies $(\forall x \in X)[f(x) \neq y]$, completing the proof.
\end{proof}

\begin{lemma}[$\rcaw$]\label{rangereversal}
The following are equivalent:
\begin{enumerate}
\item   $\prince$.
\item  If $X$ is a compact complete separable metric space,
then there is functional $R$ such that if $f\colon X \to X$ is a function with modulus of uniform
continuity $h$, $R(f,h)$ is the characteristic function of the range of $f$.
\item There is functional $R$ such that if $f\colon 2^\setN \to 2^\setN$ is a function with modulus of uniform
continuity $h$, $R(f,h)$ is the characteristic function of the range of $f$.
\end{enumerate}
\end{lemma}

\begin{proof}
We will work in $\rcaw$.  By Lemma \ref{0628A}, item (1) implies item (2).  Item~(3) is a special
case of item (2), so we need only show that item (3) implies item (1).

In $\rcaw$, we can prove the existence of the function that maps an arbitrary function $f\colon \setN \to \setN$ to an
element of Cantor space $f^\prime \colon \setN \to 2$ so that, for all $n$, $f^\prime (n) = 1$ if and only if $f(n)>0$.
In terms of the function from the definition of $\prince$, ${\sf R}_\lpo (f) = {\sf R}_\lpo (f^\prime )$.
$\rcaw$ can also prove the existence of the transformation $S\colon 2^\setN \to 2^\setN$ such that for all $f\colon \setN \to 2$,
$S(f)(n)=0$ if $(\forall m\le n)[f(m) \neq 0]$ and $S(f)(n) =1$ otherwise.  Let $\mathcal C$ denote the set of functions
from $2^\setN$ to $2^\setN$. $\rcaw$ proves the existence of the function $T\colon 2^\setN \to \mathcal C$ which maps each
$f \in 2^\setN$ to the constant function in $\mathcal C$ that takes the value $S(f)$.  For each $f$, because $T(f)$ is
a constant function, the constant $0$ function on $\setN$, denoted by $z(n) \equiv 0$,
is a modulus of uniform continuity for $T(f)$.
(Any function could be used as a modulus.)  For any $f \in \setN^\setN$, $z$ is in the range of $T(f^\prime)$ if and only
if $0$ is not in the range of $f^\prime$, and this occurs if and only if ${\sf R}_\lpo (f) = 1$.  Using the functional from
item (3), we have ${\sf R}_\lpo (f) = R(T(f^\prime ) ,z ) (z)$, so item (3) implies $\prince$.
\end{proof}  

Our proof of Banach's theorem in compact metric spaces requires a functional that can calculate the
inverse of a given function.  The next two lemmas show that $\prince$ is sufficient and also necessary for this task.

\begin{lemma}[$\rcaw + \prince$]\label{inverse}
Suppose $X$ is a compact complete separable metric space.  There is a function $I$
such that if $f\colon X \to X$ is a function with modulus of uniform continuity $h$, then $I(f,h)$ is a function that
selects elements from the pre-image of $f$.  That is, for all $y\in X$, if there is an $p$ such that$f(p)=y$,
then $f(I(f,h)(y)) = y$.
In particular, if $f$ is injective, then the restriction of $I(f,h)$ to the range of $f$ is the inverse of $f$.
\end{lemma}

\begin{proof}
Suppose that the compactness of $X$ is witnessed by the sequence of finite sequences
$\langle\langle x_{ij} : i \le n_j \rangle : j \in \setN\rangle$.
Thus for all $z\in X$, there is an $x_{ij}$ in $\langle x_{ij} : i \le n_j \rangle$ such that
$d(x_{ij} , z) < 2^{-j}$.  Given a function $f$ with a modulus of uniform continuity $h$, for
each $y$ we will calculate a rapidly converging subsequence $p= \langle p_m : m\in \setN \rangle$
such that if $y$ is in the range of $f$ then $f(p)=y$.  We will argue that this calculation is sufficiently
uniform that the desired function $I$ can be found using $\rcaw + \prince$.

Fix $f\colon X\to X$ with modulus of uniform continuity $h$, so that if $d(t_1 , t_2 ) < 2^{-h(k)}$
then $d(f(t_1), f(t_2))< 2^{-k}$.  Increasing $h$ if necessary, we may assume that $h(k) \ge k+3$ for all $k$.
Using the witness points for compactness, if $y$ is in the range of $f$, then for all $j$ we have
$(\exists k \le n_{h(j)})[ d(f(x_{k,h(j)},y) < 2^{-j}]$.

Given $f$, $h$, and $y$ as above, we can define the desired
$p = \langle p_m : m \in \setN\rangle$.  If $y$ is not in the range of $f$, let $p=y$.
If $y$ is in the range of $f$ construct
$\langle p_m : m \in \setN\rangle$ as follows.  Let $p_m = x_{ih(m)}$ where $i \le n_{h(m)}$ is the
least integer such that:
\begin{enumerate}
\item  $d(f(x_{ih(m)}),y) < 2^{-m}$,
\item  $(\forall j>m )(\exists k \le n_{h(j)} )[d(f(x_{kh(j)},y) < 2^{-j} \land d(x_{kh(j)} , x_{ih(m)})<2^{-m-2}]$, and
\item  if $m>0$, then $d(p_{m-1},x_{i h(m)} ) \le 2^{-m}$.
\end{enumerate}
The third clause ensures that $p= \langle p_m : m \in \setN \rangle$ is a rapidly converging Cauchy sequence.
By Proposition 3.6 of Kohlenbach~\cite{koh05}, $\prince$ proves that $f$ is sequentially continuous, so the
first clause shows that $f(p)=y$.  Informally, the second clause guarantees that each $p_m$ is sufficiently close
to a pre-image of $y$ that the construction can continue.  We verify this next.

To initialize the construction, we must find $p_0$.
Suppose $f(t_0 ) = y$.  Because $h(2) \ge 0+3$ we can fix an $i \le n_{h(0)}$ with $d(x_{ih(0)},t_0 )<2^{-3}$.
Because $h$ is a modulus of uniform continuity, $d(f(x_{ih(0)}),y) = d(f(x_{ih(0)} ), f(t_0 ))<2^{-0}$, so clause (1)
is satisfied.  For any $j>0$, there is a $k \le n_{h(j)}$ such that
$d(x_{kh(j)} , t_0 ) < 2^{-h(j)} < 2^{-3}$, and so
$d(f(x_{kh(j)} , y ) = d(f(x_{kh(j)} ) , f(t_0 ))< 2^{-j}$.  For such a $j$ and $k$,
\[
d(x_{kh(j)} , x_{ih(0)}) \le d(x_{kh(j)},t_0 ) + d(x_{ih(0)} , t_0 ) < 2^{-3}+2^{-3} = 2^{-2},
\]
so clause (2) is also satisfied.  The third clause is vacuously true.  We have shown that
for some $i$, $x_{ih(0)}$ satisfies all three clauses.  Let $i_0$ be the least such~$i$, and set $p_0 = x_{i_0 h(0)}$.

Suppose $p_{m-1}$ has been chosen satisfying all three clauses.
By clause (2) for $p_{m-1}$, we can find a sequence of points $\langle t_{m_j} : j \in \setN \rangle$ such that for every
$j$, $d(f(t_{m_j} ) , y ) < 2^{-j}$ and $d(t_{m_j} , p_{m-1} ) < 2^{-(m-1)-2} = 2^{-m-1}$.  In Theorem III.2.7,
Simpson~\cite{sim09} proves the generalization of the Bolzano--Weierstrass theorem for compact metric
spaces in $\ACAo$, so it is also a theorem of $\rcaw + \prince$.  Consequently, there is a subsequence of
$\langle t_{m_j} : j \in \setN \rangle$ converging to a value $t_m$ with $f(t_m ) = y$ and $d(t_m , x_{m-1} ) \le 2^{-m-1}$.
Choose $i \le n_{h(m)}$ so that $d(x_{ih(m)} , t_m ) < 2^{-h(m)} < 2^{-m-3}$.  Clause (1) holds for
$x_{ih(m)}$ because $h$ is a modulus of uniform continuity and $f(t_m ) = y$.  For any $j>m$, there is a $k \le n_{h(j)}$
such that $d(x_{kh(j)}, t_m ) < 2^{-h(j)}\le 2^{-j-3}$ and so
$d(f(x_{kh(j)} ) , y ) <2^{-j}$.  For such a $j$ and $k$,
\begin{align*}
d(x_{kh(j)}, x_{ih(m)} )  \le d(x_{kh(j)} , t_0 ) &+ d(x_{ih(m)} , t_0 ) \\
& < 2^{-j-3} + 2^{-m-3} < 2^{-m-2},
\end{align*}
so clause (2) holds for $x_{i h(m)}$.
Finally,
\[d(p_{m-1} , x_{ih(m)} ) < d(p_{m-1} , t_m ) + d (x_{ih(m)} , t_m )<2^{-m-1} + 2^{-m-3} < 2^{-m}.\]
We have shown that all three clauses hold for some choice of $i$, so let $i_m$ be the least such $i$ and
set $p_m = x_{i_m h(m)}$.  This concludes the argument that our construction never halts, yielding
the desired pre-image $p = \langle p_m : m \in \setN \rangle$.

It remains to show that $\rcaw + \prince$ suffices to prove the existence of the function $I$ from the statement
of the lemma.  Suppose we are given $f$ with modulus $h$ and a value $y$ from the metric space.   By
Lemma \ref{0628A}, $R(f,h)$ is the characteristic function for the range of $f$.  If $y$ is not in the range
of $f$, output $y$.  Otherwise, begin constructing $p$, searching for an $x_{ih(m)}$ satisfying clauses (1), (2), and (3)
above.  By $\prince$, we may use a realizer for $\lpo$ to check if clause (2) holds.  As argued above, when
$y$ is in the range, this process calculates the desired pre-image of $y$.  Summarizing, $\rcaw + \prince$ proves
the existence of the function mapping $f$, $h$, and $y$ to the desired value.  Applying $\lambda$-abstraction
yields $I(f,h)$.
\end{proof}

The use of $\prince$ in the previous lemma is necessary, as shown by the following reversal.

\begin{lemma}[$\rcaw$]\label{inversereversal}
The following are equivalent:
\begin{enumerate}
\item $\prince$.
\item  If $X$ is a compact complete separable metric space, then there is a function $I$ such that
if $f\colon X \to X$ is a function with modulus of uniform continuity $h$, then $I(f,h)$ is a function that
selects elements from the pre-image of $f$.
\item  There is a function $I$ such that if $f\colon 2^\setN \to 2^\setN$ is a function with modulus of uniform
continuity $h$, then $I(h,f)$ is a function that selects elements from the pre-image of $f$.
\end{enumerate}
\end{lemma}

\begin{proof}
We will work in $\rcaw$.  By Lemma~\ref{inverse}, item (1) implies item (2).  Item~(3) is a special case of item (2),
so we need only show that item (3) implies item (1).  By Proposition \ref{llpoimpliesE2}, it suffices to show that (3) implies
$(\llpo )$.

Given an input $w\colon \setN \to 2$ for $\llpomin$, we will show how to construct a function $f$ with modulus of uniform
continuity $h$ such that information about the pre-image of $f$ as provided by $I(f,h)$ in item (3) can be used
to calculate $\llpomin$ for  $w$.  In particular, we will control the pre-image of the constant $0$ function, denoted $\vec 0 \in 2^\setN$.
If the first $t$ where $w(t)=0$ is even, we require $f^{-1} (\vec 0 ) = \{ \vec 0 \}$. If the first $t$ such that
$w(t)=0$ is odd, we require $f^{-1} (\vec 0 ) = \{ \vec 1 \}$.  If $0$ is not in the range of $w$, $f^{-1} (\vec 0 )$ will
be the set $\{ \vec 0 , \vec 1 \}$.

Now we can specify the behavior of $f$.  Let $x \colon \setN \to 2$ by and element of $2^\setN$.  Evaluating $f$ at $x$ yields
a function $f(x)$, which also maps $\setN$ into $2$ and is defined as follows.
\begin{enumerate}
\item  $f(x)(0)=0$.
\item  For $n>0$, $f(x)(n)$ is defined by two cases:
\begin{enumerate}
\item   if $n-1$ is not the least $t$ such that $w(t)=0$ then
\[
f(x)(n)=
\begin{cases}
x(n), & {\text {if}~}x(0)=0, \\
1-x(n), &{\text {if}~}x(0)=1.
\end{cases}
\]
\item  if $n-1$ is the least $t$ such that $w(t)=0$ then
\[
f(x)(n)=
\begin{cases}
x(0), & {\text {if}~}n-1\text{~is even}, \\
1-x(0), &{\text {if}~}n-1\text{~is odd}.
\end{cases}
\]
\end{enumerate}
\end{enumerate}

Routine arguments verify that the pre-image of $f$ satisfies the requirements listed above.  Also, if the sequences
$x$ and $y$ agree in the first $n$ values, the sequences $f(x)$ and $f(y)$ also agree in the first $n$ values.
Thus, the function $h(n)=n$ is a modulus of uniform continuity for $f$.  The construction of $f$ from $w$ is
sufficiently uniform that $\rcaw$ proves the existence of a function $g$ mapping each $w \in 2^\setN$ to
its associated function $f$.

Let $I$ be the function described in item (3) of the statement of the lemma, and let $h(n)=n$ be the identity
function on $\setN$.  The for any $w \in 2^\setN$, $I(g(w),h) (\vec 0 )$ is (an element of $2^\setN$) equal
to $\vec 0$ if the first $0$ in the range of $w$ occurs at an even value, and $\vec 1$ if the first $0$
occurs at an odd value.  The sequences coding elements of $2^\setN$ output by $I$ are rapidly
converging sequences of finite approximations to $\vec 0$ or $\vec 1$.  By the definition of the metric
on $2^\setN$, the first entry in the third finite approximation for any sequence equal to $\vec 0$ will
be $0$, and similarly the value $1$ can be extracted from any sequence equal to $\vec 1$.  Thus
$I(g(w) , h)(0)$ uniformly calculates $\llpomin$ for $w$.
\end{proof}

The previous results allow us to formulate and analyze
some restrictions of Banach's theorem.
For compact complete separable metric spaces, a functional form of
Banach's theorem restricted to uniformly continuous functions
is equivalent to the functional existence principle
$\prince$.  Note that if $f$ and $g$ have moduli of uniform continuity
$h_f$ and $h_g$, then $h$ defined by $h(n)=\max \{h_f(n), h_g(n) \}$
is a modulus of uniform continuity for both $f$ and $g$.  As a notational
convenience, we will use common moduli of uniformity for pairs of functions.

\begin{definition}
For a complete separable metric space $X$, a Banach functional $B_X$ is defined as follows.
For injective functions $f\colon X \to X$ and $g\colon X \to X$ with a common modulus of uniformity $h$,
$B_X (f, g, m )$ is a bijective function $H\colon X \to X$ such that for all $x \in X$,
$H(x) = f(x)$ or $g(H(x)) = x$.  The parenthesized expression $\princb{X}$ denotes the principle asserting the
existence of a Banach functional for $X$.
\end{definition}

Following our previous pattern, the next result proves a version of Banach's theorem for compact
metric spaces using $\prince$.  The reversals and  a summary appear in a second result.

\begin{lemma}[$\rcaw + \prince$]\label{Bmetriclemma}
If $X$ is a compact metric space, then $\princb{X}$.
\end{lemma}

\begin{proof}
Assume $\rcaw$ and
$\prince$.  Suppose $X$ is a complete separable metric space and that
$\langle\langle x_{ij} : i \le n_j \rangle : j \in \setN\rangle$ witnesses that $X$ is compact.
Let $f$ and $g$ be injections of $X$ into $X$ with a common modulus of uniform continuity $h$.
Apply Lemma \ref{0628A} to find the range functionals $R(f,h)$ and $R(g,h)$.  Apply Lemma \ref{inverse} to
to find pre-image selectors $I(f,h)$ and $I(g,h)$.  Because $f$ and $g$ are injections, the restrictions of
these functions to the ranges of $f$ and $g$ are inverse functions.  Consequently, we will use the shorthand
notation $f^{-1}$ and $g^{-1}$.  Note that the pre-image selectors $f^{-1}$ and $g^{-1}$ are defined for all
inputs from $X$, and that if (for example) $y$ is in the range of $f$, then $f(f^{-1}(y)) = y$.  Our goal is
to construct the bijection $H$ in the statement of $\princb{X}$.
This is achieved by a back-and-forth construction, alternately iterating applications of $g^{-1}$ and $f^{-1}$,
and basing the value of $H$ on the terminating condition of this process.

First we construct a functional that alternately applies $g^{-1}$ and $f^{-1}$.  Using primitive recursion,
define $S(x,n)$ by $S(x,0) = x$ and for $n\ge 0$, $S(x,2n+1) = g^{-1} (S(x,2n))$ and
$S(x,2n+2) = f^{-1}(S(x,2n+1))$.  Calculating a few values yields
$S(x,0)=x$, $S(x,1) = g^{-1}(x)$, $S(x,2)=f^{-1}(g^{-1}(x))$, and $S(x,3)=g^{-1}(f^{-1}(g^{-1}(x)))$.  As noted above,
$f^{-1}$ and $g^{-1}$ are total, so $S(x,n)$ is defined for all $x$ and all $n$.

A traditional back-and-forth construction using partial inverse functions might halt if, for example,
$x$ was not in the range of $g$, or if $g^{-1}(x)$ was not in the range of $f$, and so on.
Define the function $P\colon X \times \setN \to \{0,1\}$ by $P(x,0)=1$ and for $n\ge 0$,
$P (x,2n+1) = R(g,h)(S(x,2n))$ and $P(x,2n+2) = R(f,h)(S(x,2n+1))$. Consider the initial stages
of the back-and-forth process displayed in the following table.

\begin{center}
\begin{tabular}{ l | l | l | l | l | l}
$n$=stage &0 &1 &2 & 3&$\dots$\\
\hline
$S(x,n)$ & $x$ & $g^{-1}(x)$ & $f^{-1}(g^{-1}(x))$ & $g^{-1}(f^{-1}(g^{-1}(x)))$ & $\dots$
\end{tabular}
\end{center}

\noindent
If $x$ is not in the range of $g$ then $P(x,1)=0$.  If $g^{-1}(x)$ is not in the range of $f$, then
$P(x,2)=0$.  In general, the least $n$ with $P(x,n)=0$ will be the stage where the traditional
back-and-forth process based on partial inverse functions will halt.
If the back-and-forth process does not halt, then $P(x,n)=1$ for all $n$.
Writing $P_x(n)$ for $P(x,n)$, we may view $P_x (n)$ as a function from
$\setN$ into $\{0,1\}$.  Apply $\phi$ as provided by $\prince$, and we have
$\phi(P_x (n)) = 0$ if the back-and-forth process halts and $\phi(P_x (n)) = 1$
if the process never terminates.

By $\prince$ and Theorem \ref{E2LLPO}, we may also use $R_\llpomin$, a realizer for $\llpomin$.  Define the functional
$T(x)$ by
\[
T(x) =
\begin{cases}
1, & \text{if~}\phi (P_x(n))=1,\\
R_\llpomin (P_x), &\text{if~}\phi (P_x(n))=0.
\end{cases}
\]
Finally, define the bijection $H(x)$ by
\[
H(x)=
\begin{cases}
f(x), &\text{if~}T(x)=1\text{,~and}\\
g^{-1}(x), &\text{if~}T(x)=0.
\end{cases}
\]
One can argue that $H$ is the desired bijection by the usual arguments.  Briefly, consider the
following diagram representing images and pre-images of an element $x$ from $X$.

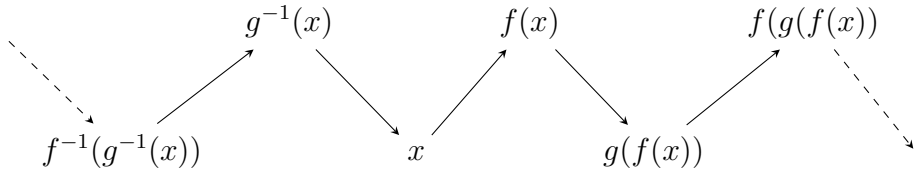
 \begin{figure}[H]
 \begin{center}
 \tikzpicture
  [mymatrix/.style={
    matrix of nodes,
    column sep=.3cm,
    row sep=1cm,
    row 2/.style={minimum height=.1cm}}
  ]
  \matrix[mymatrix] (m) {
    $ $&&$g^{-1}(x)$&&$f(x)$&&$f(g(f(x))$& \\
    &$f^{-1}(g^{-1}(x))$&&$\quad x \quad$&&$g(f(x))$&&$ $ \\
  };
  \draw[dashed,-stealth](m-1-1)--(m-2-2);
  \draw[-stealth](m-2-2)--(m-1-3);
  \draw[-stealth](m-1-3)--(m-2-4);
  \draw[-stealth](m-2-4)--(m-1-5);
  \draw[-stealth](m-1-5)--(m-2-6);
  \draw[-stealth](m-2-6)--(m-1-7);
  \draw[dashed,-stealth](m-1-7)--(m-2-8);
\endtikzpicture
\end{center}
 \end{figure}

\noindent
Each element of the lower copy of $X$ appears in at least one bipartite subgraph of the sort pictured.
Also, for each $y$ in the upper copy of $X$, we know $y = g^{-1}(g(y))$, so each element in the
upper copy of $X$ appears in at least one bipartite subgraph.  Because $f$ and $g$ are injective,
each element appears in exactly one bipartite subgraph.  The choice of the values of $H(x)$ ensure
that if the bipartite graph terminates on the left, the left most vertex is either in the lower copy of $X$
and in the domain of $H$, or in the upper copy of $X$ and in the range of $H$.  Thus $H$ is a bijection
of $X$ into itself, satisfying the requirements of $\princb{X}$.
\end{proof}

This section concludes with proofs of two reversals for instances of the previous lemma, summarizing
the results for Banach's theorem on compact metric spaces in the following theorem.

\begin{theorem}[$\rcaw$]\label{Bmetric}
The following are equivalent:
\begin{enumerate}
\item  $\prince$\label{BmetricA}.
\item  If $X$ is a compact metric space, then $\princb{X}$.\label{BmetricB}
\item  $\princb{[0,1]}$\label{BmetricC}.
\item  $\princb{2^\setN}$\label{BmetricD}.
\end{enumerate}
\end{theorem}

\begin{proof}
The previous lemma proves that item (\ref{BmetricA}) implies item (\ref{BmetricB}).  
Item (\ref{BmetricC}) and item (\ref{BmetricD}) are special cases of item~(\ref{BmetricB}),
so we can complete the proof by reversing (\ref{BmetricC}) and (\ref{BmetricD}) to (\ref{BmetricA}).
For the first reversal, suppose $B_{[0,1]}(f,g,h)$ is the Banach functional for $[0,1]$.
Consider the injections $f$ and $g$ defined by $f(x)=g(x)=x/2$.  Each $x$ in $[0,1]$ is
represented by a
rapidly converging sequence of rationals, and dividing each element of the sequence by $2$ yields
a rapidly converging sequence representing $x/2$.   Thus $\rcaw$ proves that $f$ and $g$ are defined
and total.  The identity function $h(k)=k$ is a modulus of uniform continuity for $f$ and $g$.
Suppose $H=B_{[0,1]}(f,g,h)$ is the bijection satisfying Banach's theorem for $f$ and $g$.
Consider $x= \frac{1}{2}$ and the sequence $x_n = \frac{1}{2} + \frac{1}{2^n}$.  For each $n$,
$x_n$ is not in the range of $g$,  so $H(x_n) = f(x_n ) = \frac{x_n}{2} = \frac{1}{4} + \frac{1}{2^{n+1}}$.
Thus, $\lim_{n \to \infty} H(x_n) = \frac{1}{4}$.  The functional $H$ is bijective, so $1$ is in the range of $H$.
Fix $x$ with $H(x)=1$.  By the Banach theorem, $H(x)=f(x)$ or $H(x) = g^{-1}(x)$.  Because $1$ is not in the
range of $F$, $H(x) = g^{-1} (x)$.  Thus $1 = g^{-1} (x)$, so $x= \frac{1}{2}$ and $H(\frac{1}{2}) = 1$.
Summarizing,
\[
H(\lim_{n \to \infty} x_n ) = H (\frac{1}{2}) = 1 \neq \frac{1}{4} = \lim _{n \to \infty} H (x_n ).
\]
Thus $H$ is not sequentially continuous at $x = \frac{1}{2}$, and
$\prince$ follows by Proposition~\ref{K3.7}.

For the final reversal, suppose $B_{2^\setN}(f,g,h)$ is the Banach functional for Cantor space.
Consider the padding function $P(x)$ that adds a zero after each entry in a binary input string.
Formally, $P(x)(n) = x(m)$ if $n=2m$, and $P(x)(n) = 0$ otherwise.  For example,
\[
P(\langle 1,0,1,1\dots \rangle ) = \langle 1, 0,0 ,0, 1,0,1,0\dots\rangle.
\]
$\rcaw$ proves that $P(x)$ is defined and total, and that the identity function $h(k)=k$ is
a modulus of uniform continuity for $P(x)$.  Let $f(x)$ and $g(x)$ both be $P(x)$.
Let $H = B_{2^\setN}(f,g,h)$ be the bijection satisfying Banach's theorem for $f$ and $g$.
For each $n$, let $\sigma_n$ consist of $n$ copies of the string $10$, followed by $11$, followed by zeros.
The double 1 ensures that $\sigma_n$ is not in the range of $g(x)=P(x)$.  Thus, for each $n$,
$H(\sigma_n ) = f(\sigma_n ) = P(\sigma_ n)$, which consists of $n$ copies of the string $1000$ followed
by $1010$, followed by zeros.  Thus $\lim_{n\to \infty} H(\sigma_n )$ is the string $1000$ repeated infinitely.
On the other hand, $\lim_{n \to \infty} \sigma_n $ is $\langle 1, 0, 1, 0  \dots \rangle$.  The string
$\langle 1, 1, 1 \dots \rangle$ is not in the range of $f(x)$, so $H(g(\langle 1, 1, 1 \dots \rangle))=\langle 1, 1, 1 \dots \rangle$.
Because $g(\langle 1, 1, 1  \dots \rangle)=\langle 1, 0, 1, 0   \dots \rangle$, we have
$H(\lim_{n\to \infty} \sigma_n ) = H( \langle 1, 0 ,1, 0  \dots \rangle) = \langle 1, 1, 1  \dots \rangle$.
Thus $H(\lim_{n\to \infty} \sigma_n ) \neq \lim_{n \to \infty} H(\sigma_n)$, so $H$ is not
sequentially continuous at $x = \langle 1,0,1,0   \dots\rangle$.  The principle
$\prince$ follows by Proposition~\ref{K3.7}, completing the reversal and the proof of the theorem.
\end{proof}

We note that the functional $R$ in Lemma~\ref{0628A}, the functional $I$ in Lemma~\ref{inversereversal}, and the functional $B$ in Theorem~\ref{Bmetric} are constructed uniformly in a code for the space $X$. Hence these functionals could, in principle, be defined with $X$ as a parameter. This is another layer of uniformity in the constructions, although noting the parameter explicitly complicates the notation.

\section{Moduli of uniform continuity}\label{sec:moduli}

This section introduces a function that computes moduli of uniform continuity.  As shown below, the strength of the existence of the function lies below $\prince$, allowing us to streamline the definition of Banach functionals and Theorem~\ref{Bmetric}.

\begin{definition}
Suppose $X$ is a compact complete separable metric space and $Y$ is a complete separable metric
space.  The principle $( \sf M )$ asserts the existence of a function $M$ such that if $f\colon X \to Y$ is continuous,
then $M(f)$ is a modulus of uniform continuity for $f$.
\end{definition}

Near the end of his article, Kohlenbach~\cite{koh05} presents a functional form of the fan theorem, denoted
by $({\sf MUC})$.  He notes that $( \sf M )$ is a consequence of $(\sf{MUC} )$, $\sf{MUC}$ is conservative
over $\WKLo$ for second order sentences, and $(\sf{MUC})$ is inconsistent with $\prince$.
Because $(\sf{MUC})$ proves $(\sf M )$, $(\sf M )$ is also conservative over $\WKLo$ for second order sentences.
The next lemma shows
that unlike $( \sf{MUC} )$, the principle $(\sf M )$ is a consequence of $\prince$.

\begin{lemma}[$\rcaw$]\label{7B1}
$\prince$ implies $(\sf M )$.
\end{lemma}

\begin{proof}
Let $X$ be a compact complete separable metric space with compactness witnessed by the sequence of sequences
$\langle \langle x_{ij} : i \le n_j \rangle : j \in \setN \rangle$.  Let $Y$ be a comple separable metric space.
We will use $d$ to denote the metric in both spaces.
For $f\colon X \to Y$ we can define a prospective value of a modulus of uniform continuity for $f$ at $m$ by setting
$(M(f))(m)$ equal to the least $n$ such that:
\begin{equation}\label{modulus}
(\forall x_{ij})( \forall x_{i^\prime j^\prime })
[ d(x_{ij} , x_{i^\prime j^\prime } ) < 2^{-n} \to
d(f(x_{ij} ), f(x_{i^\prime j^\prime }))<2^{-m-1} ]
\end{equation}
Informally, $M(f)$ is a function from $\setN$ to $\setN$ that resembles a modulus of uniform continuity on
the compactness witnesses for $X$.  First we will show that $\rcaw + \prince$ suffices to prove the existence
of the function $M$.  Then we will verify that if $f$ is continuous, then $M(f)$ is a modulus of uniform
continuity for $f$.

Working in $\rcaw + \prince$, let $X$ and $Y$ be as above, and suppose $f\colon X \to Y$.
Recalling the reverse mathematical formalization of inequalities in the reals,
the formulas
$d(x_{ij} , x_{i^\prime j^\prime })<2^{-n}$ and
$d(f(x_{ij}),f(x_{i^\prime j^\prime }))>2^{-m-1}$ are
$\Sigma^0_1$.  Thus $\rcaw$ proves the existence of a function
$a(f,m,n,t)$ which is $0$ if $t$ codes a witness that there are $x_{ij}$ and $x_{i^\prime j^\prime }$
such that $d(x_{ij} , x_{i^\prime j^\prime })<2^{-n}$ and
$d(f(x_{ij}),f(x_{i^\prime j^\prime }))>2^{-m-1}$, and is $1$ otherwise.  Note that formula (\ref{modulus}) holds
if $a(f,m,n,t)$ is $1$ for all $t$, and fails if there is a $t$ such that $a(f,m,n,t)$ is $0$.
As noted in section \ref{sec:realize}, $\prince$ implies the existence of the function $R_\lpo$.
The $\lambda$ notation $\lambda t . a(f,m,n,t)$ denotes the function that maps each $t\in \setN$
to the value $a(f,m,n,t)$.  Applying $\lambda$ abstraction (which is a consequence of $\rcaw$~\cite{koh05})
and $\prince$, we can prove the existence of the function
$b(f,m,n) = R_\lpo (\lambda t . a(f,m,n,t ))$.  Note that
for all $f$, $m$, and $n$, $b(f,m,n)=1$ if formula (\ref{modulus}) holds and
$b(f,m,n) = 0$ otherwise.  By Proposition \ref{K3.9}, $\prince$ proves the
existence of Feferman's $\mu$, so by $\prince$ and an additional application of $\lambda$ abstraction,
we can prove the existence of the function
$c(f,m) = \mu ( 1 - \lambda n. b(f,m,n))$.  Note that for each $f$ and $m$, if there is an $n$ such that
formula (\ref{modulus}) holds, then $c(f,m)$ is the least such $n$.  If there is no such $n$, for example if $f$
is discontinuous, then $c(f,m)$ still yields some value, but no useful information.
By $\lambda$ abstraction, $\rcaw + \prince$ proves the existence of
$M(f) = \lambda m . c(f,m)$.  For every $f\colon X \to Y$, $M(f)$ yields a function from $\setN$ to $\setN$.

It remains to show that if $f$ is continuous then $M(f)$ is a modulus of uniform continuity for $f$.
Fix a continuous $f\colon X \to Y$ and $m \in \setN$.  Let $n=M(f)(m)$.  Suppose that $u,v\in X$ satisfy
$d(u,v) < 2^{-n}$.  Choose $\delta < 2^{-n} - d(u,v)$.  Because $f$ is continuous and
$\langle \langle x_{ij} : i \le n_j \rangle : j \in \setN \rangle$ is dense in $X$, we can find
an $x_{ij}$ such that $d(x_{ij} , u ) < \delta/2$ and
$d(f(x_{ij}) , f(u)) < 2^{-m-2}$.  Similarly, find
$x_{i^\prime j^\prime}$ such that
$d(x_{i^\prime j^\prime}, v) < \delta/2$ and
$d(f(x_{i^\prime j^\prime}),f(v))<2^{-m-2}$.
By the triangle inequality,
\[
d(x_{ij}, x_{i^\prime j^\prime} ) \le d(x_{ij}, u) + d(u,v) + d(v, x_{i^\prime j^\prime})
<\delta/2 + d(u,v) + \delta/2 < 2^{-n}.
\]
Because $d(x_{ij}, x_{i^\prime j^\prime} ) <2^{-n}$,
and because $(M(f))(m)=n$, formula (\ref{modulus}) holds, so
$d(f(x_{ij} ) , f( x_{i^\prime j^\prime} ))<2^{-m-1}$.  By the triangle inequality,
\[
\begin{aligned}
d(f(u),f(v))&< d(f(u),f(x_{ij})) + d(f(x_{ij} , f(x_{i^\prime j^\prime} ) ) + d(f(x_{i^\prime j^\prime}),f(v))\\
 &< 2^{-m-2} + 2^{-m-1} + 2^{-m-2} = 2^{-m}.
\end{aligned}
\]
Summarizing, when $f$ is continuous and $M(f)(m) = n$, if $d(u,v) < 2^{-n}$ then
$d(f(u),f(v)<2^{-m}$.  Thus $M(f)$ is a modulus of uniform continuity for $f$.
\end{proof}

The principle $(\sf M )$ allows us to reformulate Theorem \ref{Bmetric}, stripping all reference to moduli of uniform continuity.

\begin{theorem}[$\rcaw$]\label{7C1}
 The principle $\prince$ is equivalent to the statement that for every compact complete
separable metric space $X$, there is a function $B^\prime_X$ that maps each pair of injections
from $X$ to $X$ to a bijection satisfying Banach's theorem.
\end{theorem}

\begin{proof}
Assuming $\prince$, by Lemma \ref{7B1} we may use the function $M$ to calculate moduli of uniform
continuity for $f$ and $g$.  The pointwise maximum function $\max (M(f), M(g) )$ is a joint modulus of
uniform continuity for $f$ and $g$.  If $B_X (f,g,m)$ is the function provided by Theorem \ref{Bmetric} part (2),
then the function defined by $B^\prime_X (f,g) = B_x(f,g,\max (M(f), M(g) ))$ is the desired Banach function.
The converse is
immediate from Theorem \ref{Bmetric}.
\end{proof}

Because $(\sf M )$ is a consequence of $(\sf{MUC})$, the principle $(\sf M )$ does not imply
$\prince$.  That is, the converse of Lemma \ref{7B1} is not true.  The next two results show that like $(\sf{MUC})$,
the second order theorems of $(\sf M )$ are exactly those of $\WKLo$.  As part of that proof,
the next lemma allows us to change representations of functions, with the eventual goal of applying a
traditional reverse mathematics result to show that $(\sf M )$ implies $\WKLo$.

\begin{lemma}[$\rcaw$]\label{7E1}
 Suppose $X$ and $Y$ are complete separable metric spaces.
Suppose that $\Phi$ is a code for a totally defined continuous function
as described in Definition II.6.1 of Simpson~\cite{sim09}.  Then there is a function
$f\colon X \to Y$ such that for all $n$, $a$, $r$, $b$, and $s$, if $(n,a,r,b,s) \in \Phi$ then
$d(f(a),b) \le s$.
\end{lemma}

\begin{proof}
Working in $\rcaw$, suppose $X$, $Y$, and $\Phi$ are as above.  Fix $x \in X$.
Because $x$ is in the domain of the function defined by $\Phi$, for each $m$ we can
find $(n,a,r,b,s) \in \Phi$ (occurring first in some fixed enumeration of quintuples)
such that $d(x,a) < r$ and $s < 2^{-m-1}$.  Set $f(x)(m) = b$.  The sequence
$\langle f(x)(m) : m \in \setN \rangle$ is a rapidly converging sequence of elements of
$Y$ converging to the desired value of $f(x)$.  $\rcaw$ proves the existence of $f$.

We now verify the last sentence of the lemma.  Suppose $(n,a,r,b,s) \in \Phi$.  Let $\varepsilon >0$
and choose $m$ so that $2^{-m-1} < \min \{ \varepsilon /2 , s \}$.  Let
$(n^\prime , a^\prime , r^\prime , b^\prime , s^\prime ) \in \Phi$ be the quintuple witnessing the
the value for $f(a)(m)$.  Then $d(a,a^\prime)<r^\prime$ and
$s^\prime < 2^{-m-1} < \varepsilon /2$.  Let $r_0 = \min \{ r , r^\prime - d(a,a^\prime ) \}$.
Then the ball $B(a, r_0 )$ is a subset of $B(a,r)$, and is also a subset of $B(a^\prime, r^\prime )$.
Applying property (2) of Simpson's Definition II.6.1, we have
$(a, r_0 ) \Phi (b,s)$ and $(a,r_0) \Phi (b^\prime , s^\prime)$.  By property (1) of Simpson's definition,
$d(b,b^\prime ) \le s+s^\prime < s+ \varepsilon/2$.  By the choice of $m$,
$d(b^\prime , f(a) )\le 2^{-m} < \varepsilon /2$.  By the triangle inequality
$d(f(a) ,b ) < s+ \varepsilon$.  Because $\varepsilon$ was an arbitrary positive value,
$d(f(a) ,b) \le s$.
\end{proof}

The preceding lemma allows us to completely characterize the second order theory of $(\sf M )$.

\begin{proposition}\label{7F1}
The second order theorems of $\rcaw + ( \sf M )$ are exactly the same as those of $\WKLo$.
\end{proposition}

\begin{proof}
As noted before,
$(\sf M )$ is a consequence of Kohlenbach's $(\sf{MUC})$, and so any second order theorem
provable using $(\sf M )$ is provable in $\WKLo$.  It remains to show that $(\sf M )$ implies $\WKLo$.
By Theorem IV.2.3 of Simpson~\cite{sim09}, it suffices to show that if $f$ is a continuous function
(coded by $\Phi$) on $[0,1]$, then $f$ is uniformly continuous.   Suppose $\Phi$ codes a continuous
function on $[0,1]$.  By Lemma \ref{7E1}, $\rcaw$ proves that there is a function $f\colon [0,1] \to \mathbb R$
matching the values of the coded function.  Applying $(\sf M )$, the function $M(f)$ is a modulus of
uniform continuity for $f$, and so also for the function coded by $\Phi$.  Thus $\Phi$ codes a uniformly
continuous function on $[0,1]$.
\end{proof}

We conclude by pointing out the potential and limitations of this section.
The principle $(\sf M )$ can be viewed as a higher order analogue of $\WKLo$ in much
the same fashion that $\prince$ is a higher order analogue of $\ACAo$.  A number of
Skolemized forms of statements equivalent to $\WKLo$ may be equivalent to $(\sf M )$
over $\rcaw$.  (But not all, as witnessed by Kohlenbach's $\sf {UWKL}$.  See Proposition \ref{E2WKL}.)
However, $(\sf M )$ may not be the only reasonable candidate for a $\WKLo$ analogue.
For example, reformulating
$( \sf M )$ as a function on second order continuous function codes yields an alternative
principle $(\sf M _c )$.  It seems likely that Proposition \ref{7F1} holds for $(\sf M _c )$,
but it is possible that neither $(\sf M )$ nor $(\sf M _c )$ proves the other over $\rcaw$.


\bibsection


\begin{biblist}[\normalsize]

\bib{banach}{article}{
      author={Banach, Stefan},
      title={Un th\'eor\`eme sur les transformations biunivoques},
      journal={Fundamenta Mathematicae},
      volume={6},
      date={1924},
      pages={236--239},
 }


\bib{bg-2011}{article}{
   author={Brattka, Vasco},
   author={Gherardi, Guido},
   title={Weihrauch degrees, omniscience principles and weak computability},
   journal={J. Symbolic Logic},
   volume={76},
   date={2011},
   number={1},
   pages={143--176},
   issn={0022-4812},
   review={\MR{2791341}},
   doi={10.2178/jsl/1294170993},
}
	

\bib{MR4472209}{book}{
   author={Dzhafarov, Damir D.},
   author={Mummert, Carl},
   title={Reverse mathematics: problems, reductions, and proofs},
   series={Theory and Applications of Computability},
   publisher={Springer, Cham},
   date={2022},
   pages={xix+488},
   isbn={978-3-031-11366-6},
   isbn={978-3-031-11367-3},
   review={\MR{4472209}},
   doi={10.1007/978-3-031-11367-3},
}

 \bib{fefhbk}{article}{
   author={Feferman, Solomon},
   title={Theories of finite type related to mathematical practice},
   conference={
      title={Handbook of mathematical logic},
   },
   book={
      series={Stud. Logic Found. Math.},
      volume={90},
      publisher={North-Holland, Amsterdam},
   },
   date={1977},
   pages={913--971},
   review={\MR{3727428}},
}

\bib{grilliot}{article}{
   author={Grilliot, Thomas J.},
   title={On effectively discontinuous type-$2$ objects},
   journal={J. Symbolic Logic},
   volume={36},
   date={1971},
   pages={245--248},
   issn={0022-4812},
   review={\MR{290972}},
   doi={10.2307/2270259},
}

\bib{hirstthesis}{thesis}{
   author={Hirst, Jeffry L.},
   title={Combinatorics in subsystems of second order arithmetic},
   type={Ph.D. thesis, The Pennsylvania State University},
   publisher={ProQuest LLC, Ann Arbor, MI},
   date={1987},
   pages={153},
   review={\MR{2635978}},
}

\bib{hirstmt}{article}{
   author={Hirst, Jeffry L.},
   title={Marriage theorems and reverse mathematics},
   conference={
      title={Logic and computation},
      address={Pittsburgh, PA},
      date={1987},
   },
   book={
      series={Contemp. Math.},
      volume={106},
      publisher={Amer. Math. Soc., Providence, RI},
   },
   date={1990},
   pages={181--196},
   review={\MR{1057822}},
   doi={10.1090/conm/106/1057822},
}


\bib{kohWKL}{article}{
   author={Kohlenbach, Ulrich},
   title={On uniform weak K\"{o}nig's lemma},
   note={Commemorative Symposium Dedicated to Anne S. Troelstra
   (Noordwijkerhout, 1999)},
   journal={Ann. Pure Appl. Logic},
   volume={114},
   date={2002},
   number={1-3},
   pages={103--116},
   issn={0168-0072},
   review={\MR{1879410}},
   doi={10.1016/S0168-0072(01)00077-X},
}
      
\bib{koh05}{article}{
   author={Kohlenbach, Ulrich},
   title={Higher order reverse mathematics},
   conference={
      title={Reverse mathematics 2001},
   },
   book={
      series={Lect. Notes Log.},
      volume={21},
      publisher={Assoc. Symbol. Logic, La Jolla, CA},
   },
   date={2005},
   pages={281--295},
   review={\MR{2185441}},
}



\bib{ns-2020}{article}{
   author={Normann, Dag},
   author={Sanders, Sam},
   title={On the uncountability of $\mathbb{R}$},
   journal={J. Symb. Log.},
   volume={87},
   date={2022},
   number={4},
   pages={1474--1521},
   issn={0022-4812},
   review={\MR{4510829}},
   doi={10.1017/jsl.2022.27},
}



\bib{ns-2022}{article}{
   author={Normann, Dag},
   author={Sanders, Sam},
   title={On robust theorems due to Bolzano, Weierstrass, Jordan, and Cantor},
   journal={J. Symb. Log.},
   date={2022},
   note={to appear}
}




\bib{remmel}{article}{
   author={Remmel, J. B.},
   title={On the effectiveness of the Schr\"{o}der--Bernstein theorem},
   journal={Proc. Amer. Math. Soc.},
   volume={83},
   date={1981},
   number={2},
   pages={379--386},
   issn={0002-9939},
   review={\MR{624936}},
   doi={10.2307/2043533},
}

\bib{sy}{article}{
   author={Sakamoto, Nobuyuki},
   author={Yamazaki, Takeshi},
   title={Uniform versions of some axioms of second order arithmetic},
   journal={MLQ Math. Log. Q.},
   volume={50},
   date={2004},
   number={6},
   pages={587--593},
   issn={0942-5616},
   review={\MR{2096172}},
   doi={10.1002/malq.200310122},
}

\bib{sim09}{book}{
   author={Simpson, Stephen G.},
   title={Subsystems of second order arithmetic},
   series={Perspectives in Logic},
   edition={2},
   publisher={Cambridge University Press, Cambridge; Association for
   Symbolic Logic, Poughkeepsie, NY},
   date={2009},
   pages={xvi+444},
   isbn={978-0-521-88439-6},
   review={\MR{2517689}},
   doi={10.1017/CBO9780511581007},
}

\bib{W-1992}{article}{
  author={Klaus Weihrauch},
  title={The TTE-Interpretation of Three
    Hierarchies of Omniscience Principles},
  journal={Informatik Berichte},
  volume={130},
  publisher={FernUniversit\"at Hagen},
  date={September 1992},
  }

\end{biblist}
\end{document}